\documentclass[12pt]{amsart}

\usepackage{graphicx}        
\usepackage{color}
\usepackage{amssymb,amsmath}
\usepackage{bbm}

\newtheorem{theorem}{Theorem}     
\numberwithin{theorem}{section}
\newtheorem{lemma}[theorem]{Lemma}     
\newtheorem{corollary}[theorem]{Corollary}     
\newtheorem{proposition}[theorem]{Proposition}     

\theoremstyle{definition}  
     
\newtheorem{example}[theorem]{Example}     
\newtheorem{remark}[theorem]{Remark}     

\def\cprime{$'$}
\newcommand{\CC}{{\mathbb C}}

\newcommand{\NN}{{\mathbb N}}
\newcommand{\QQ}{{\mathbb Q}}
\newcommand{\RR}{{\mathbb R}}
\newcommand{\TT}{{\mathbb T}}
\newcommand{\ZZ}{{\mathbb Z}}

\newcommand{\bbI}{\mathbbm{1}}

\newcommand{\calA}{{\mathcal A}}
\newcommand{\calB}{{\mathcal B}}
\newcommand{\calF}{{\mathcal F}}
\newcommand{\calG}{{\mathcal G}}
\newcommand{\calM}{{\mathcal M}}
\newcommand{\calO}{{\mathcal O}}
\newcommand{\calT}{{\mathcal T}}

\DeclareMathOperator{\closed}{closed}
\DeclareMathOperator{\conv}{conv}
\DeclareMathOperator{\cone}{cone}
\DeclareMathOperator{\spec}{spec}
\DeclareMathOperator{\supp}{supp}
\DeclareMathOperator{\Star}{star}
\DeclareMathOperator{\Hom}{Hom}
\DeclareMathOperator{\mon}{mon}
\newcommand{\lhra}{\ensuremath{\lhook\joinrel\relbar\joinrel\rightarrow}}

\newcommand{\simplex}{\includegraphics{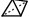}}
\newcommand{\bsimplex}{\includegraphics{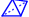}}

\definecolor{navyBlue}{cmyk}{1,1,0,0.2}
\newcommand{\defcolor}[1]{{\color{blue} #1}}
\newcommand{\demph}[1]{{\sl\defcolor{#1}}}
\newcounter{FNC}[page]
\def\fauxfootnote#1{{\addtocounter{FNC}{2}$^\fnsymbol{FNC}$%
    \let\thefootnote\relax\footnotetext{\color{magenta}{$^\fnsymbol{FNC}$#1}}}}

\title{Irrational toric varieties and secondary polytopes}
\author[A.~Pir]{Ata Firat Pir}     
\address{Ata Firat Pir\\     
  3200 Payne Ave, Apt. 24\\
  San Jose, CA \ 95117}     
\email{atafirat@gmail.com}     
\author[F.~Sottile]{Frank Sottile}     
\address{Frank Sottile\\     
         Department of Mathematics\\     
         Texas A\&M University\\     
         College Station\\     
         Texas \ 77843\\     
         USA}     
\email{sottile@math.tamu.edu}     
\urladdr{http://www.math.tamu.edu/~sottile}
\thanks{Research of Pir and Sottile supported in part by NSF grant DMS-1501370.}
\thanks{This manuscript has no associated data.}
\begin{document}

\begin{abstract}
 The space of torus translations and degenerations of a projective toric variety forms a toric variety associated to the secondary fan of  
 the {\sl integer} points in the polytope corresponding to the toric variety.
 This is used to identify a moduli space of real degenerations with the secondary polytope.
 A configuration $\calA$ of real vectors gives an irrational projective toric variety in a simplex.
 We identify a space of translations and degenerations of the irrational projective toric variety with the secondary polytope of
 $\calA$.  
 For this, we  develop a theory of irrational toric varieties associated to arbitrary fans.
 When the fan is rational, the irrational toric variety is the nonnegative part of the corresponding classical
 toric variety. 
 When the fan is the normal fan of a polytope, the irrational toric variety is homeomorphic to that polytope. 
\end{abstract}

\maketitle

\section*{Introduction}

Fulton observed that ``Toric varieties have provided a remarkably fertile testing ground for general theories''~\cite[Preface,~ix]{Fulton}.
One reason for this is their connection to geometric combinatorics.
Toric varieties in algebraic geometry arise in three guises~\cite[Chs.~1--3]{CLS}:
as an affine variety $X_\calA$ parametrized by exponents $\calA\subset\ZZ^n$ of monomials, as a variety $X_P$ projectively embedded by a
line bundle associated to an integer polytope $P\subset\RR^n$, and as a normal variety $X_\Sigma$ functorially constructed from a rational
fan $\Sigma\subset\RR^n$. 
These three guises come together in a beautiful solution to a moduli problem.
If we replace $X_\calA$ by its projective closure, the torus of projective space acts on $X_\calA$ and the collection of torus translates
of $X_\calA$ and their limit schemes forms a torus-invariant subscheme of the Hilbert scheme.
Its normalization is the toric variety $X_{\Sigma}$ associated to the secondary fan $\Sigma$ of $\calA$~\cite{Alexeev,KSZ1,KSZ2}.
Embedding $X_{\Sigma}$ using the secondary polytope $P$ of $\calA$ identifies $X_\Sigma$ with $X_P$, and this has a moment map to $P$, which
identifies its nonnegative part with $P$.

Toric varieties, or at least their nonnegative parts, occur naturally in applications of mathematics.
In statistics, an exponential family (of probability distributions) is the nonnegative part of the projective cloure of an affine toric
variety $X_\calA$, considered as a subset of the probability simplex.
These arose in the 1930's~\cite{Dar,Koopman,PW} and are now called toric models~\cite[Sect.~1.2.2]{ASCB} or  log-linear
models~\cite{Goodman}. 
A version of the moment map (sufficient statistics) is important for toric models, and in 1963~\cite{Birch} Birch proved that the toric
model is homeomorphic to the polytope $\conv(\calA)$ under this `algebraic moment map'.
This predates the celebrated (and more general) work of Atiyah~\cite{A82,GS82} on moment maps.
Projective toric varieties $Y_P$ are also understood as the source for B\'ezier patches in geometric modeling.
This is a consequence of Krasauskas's introduction of toric B\'ezier patches~\cite{Krasauskas}, which generalize classical B\'ezier patches.
For these, the moment map (an independent proof is given in~\cite[Thm.~26]{Krasauskas}) underlies the important property of linear
precision~\cite{LPPP}. 

It is also natural to consider toric patches as arising from the nonnegative part of an affine toric variety $X_\calA$.
In this guise, both applications allow the exponents $\calA$ to be any real vectors.
This leads to irrational affine toric varieties $Y_\calA$~\cite{CGS}, which are analytic subsets of the nonnegative orthant of a real vector
space.
When $\calA\subset\ZZ^n$, $Y_\calA$ is the nonegative part of the affine toric variety $X_\calA$.
When $\calA$ lies on an affine hyperplane, the intersection $Z_\calA$ of $Y_\calA$ with the standard (probability) simplex is both an
irrational toric model and a source for toric B\'ezier patches.
By Birch's Theorem, $Z_\calA$  is homeomorphic to the polytope $\conv(\calA)$.

Applications from geometric modeling~\cite{CGS,GSZ} lead to the irrational version of the moduli problem
involving the secondary fan and polytope.
That is, there is a multiplicative action of the positive torus on the probability simplex, and understanding the possible limits of
translates of $Z_\calA$ provides an explanation of control structures for patches.
Here, these limits are understood set-theoretically or rather metrically in terms of the Hausdorff topology on closed subsets of the
probability simplex.
When $\calA\subset\ZZ^n$, the toric moduli spaces from~\cite{Alexeev,KSZ1,KSZ2} imply that the possible limits are all toric degenerations
of $Z_\calA$, and identifies the real points of the moduli space with the secondary polytope of $\calA$.
This is the main result in~\cite{GSZ}.

When $\calA$ is irrational, we do not have the full power of algebraic geometry, and other means are needed to study the space of real
torus translates of $Z_\calA$ and their Hausdorff limits.
In~\cite{PSV}, all Hausdorff limits are shown to be toric degenerations of $Z_\calA$ and the space of translates and limits is
understood set-theoretically in terms of the secondary fan of $\calA$.
Our purpose is to promote this set-theoretic understanding to one in terms of equivariant cell complexes and show that the resulting moduli
space (defined in Section~\ref{S:Hausdorff}) is  
homeomorphic to the secondary polytope of $\calA$, this is the content of Theorem~\ref{Th:SecondaryPolytope}.

To accomplish this identification, we develop a theory of irrational toric varieties $Y_\Sigma$ associated to arbitrary fans $\Sigma$ in
real vector spaces.
This theory is very satisfying, with many parallels to the classical theory of toric varieties associated to rational fans.
The irrational toric variety $Y_\Sigma$ is an equivariant cell complex (Theorem~\ref{Th:ITV_structure}), the association from fans is
functorial (Theorem~\ref{Th:mapsOfFans}), and the fan $\Sigma$ may be recovered from the  irrational toric variety $Y_\Sigma$
(Corollary~\ref{C:recoverFan}). 
Also, $Y_\sigma$ is compact if and only if the fan $\Sigma$ is complete (Theorem~\ref{Th:complete}).

An important property of classical toric varieties that we do not yet have for irrational toric varieties is an equivalence of categories
between fans and irrational toric varieties.
We believe this will require enriching irrational toric varieties and irrational fans with additional structure, so that we obtain a sheaf
of functions on an irrational toric variety possessing some form of noetherianity.
One possible source for these structures, at least locally, may be the recent articles of Miller~\cite{Miller3,Miller2,Miller1}.

This theory of irrational toric varieties has a similar motivation to the development of small covers~\cite{DJ} from toric
topology~\cite{BP}, noncommutative toric varieties~\cite{KLMV}, and Ford's toroidal embedding from an irrational
fan~\cite{Ford}---these all generalize some aspects of toric varieties to more general polytopes and fans.
We also understand LVM manifolds~\cite{LV,Me} as another topic in this theme.
Some connections between these topics are covered in the survey~\cite{Verjovsky}.
Another related approach to generalizing classical toric varieties to irrational objects from the view of symplectic geometry are
tori quasifolds~\cite{BaPr,BZ,P}.
A forthcoming volume of surveys~\cite{IMSA_volume} aims to help unify these different irrational generalizations of classical toric
varieties. 

In Section~\ref{S:classical}, we sketch the classical construction of a toric variety from a rational fan, and 
recall that a toric variety is a functor on commutative monoids.
We recall some properties of irrational affine toric varieties in Section~\ref{S:Affine}.
We construct irrational toric varieties from arbitrary fans in Section~\ref{S:irrational}, and 
establish their main properties.
Section~\ref{S:projective} develops global properties of irrational toric varieties.
While our results on irrational toric varieties parallel some on classical toric varieties, their proofs require
different methods, as fundamental facts from algebraic geometry do not hold for irrational toric varieties.
This is condensed from the 2018 Texas A\&M Ph.D.\ thesis of Pir~\cite{Ata_Thesis}.
Having developed this theory, we use it in Section~\ref{S:Hausdorff} to establish our main result,
identifying the moduli space of Hausdorff limits of a projective irrational toric variety with the secondary polytope.

\section{Classical Toric Varieties}
\label{S:classical}

We review the construction of toric varieties in algebraic geometry from rational fans and recall some of their
properties. 
For additional treatment of toric varieties, see any of~\cite{CLS,Ewald,Fulton}.
For more on geometric combinatorics, see~\cite{GKZ,GBCP,Ziegler}.

Let \defcolor{$N_\ZZ$} be a free abelian group of rank $n$ ($N_\ZZ\simeq\ZZ^n$) and let $\defcolor{M_\ZZ}:=\Hom(N_\ZZ,\ZZ)$ be
its dual group.
Write $\defcolor{u\cdot v}\in\ZZ$ for the pairing, where $u\in M_\ZZ$ and $v\in N_\ZZ$.
Write $\TT_N$ for the abelian group scheme $\spec \ZZ[M_\ZZ]$.
This is a torus (integral abelian group scheme) as $\ZZ[M_\ZZ]$ is a domain.
Its lattice of cocharacters is $N_\ZZ$ and $M_\ZZ$ is its lattice of characters. 

Let \defcolor{$\RR_>$} be the positive real numbers and \defcolor{$\RR_\geq$} be the nonnegative real numbers.
Let \defcolor{$N$} be the real vector space $N_\ZZ\otimes_\ZZ\RR$ ($\simeq\RR^n$).
A (polyhedral) cone $\sigma\subset N$ is a submonoid of the form
 \begin{equation}\label{Eq:cone}
   \defcolor{\cone\{v_1,v_2,\dotsc, v_k\}}\ :=\   
   \RR_\geq v_1 \ +\  \RR_\geq v_2 \ +\ \dotsb \ +\   \RR_\geq v_k \,,
 \end{equation}
where $v_1,\dotsc,v_k\in N$.
(A \demph{monoid} is a set endowed with an associative commutative operation with identity.)
The \demph{dual cone} of a polyhedral cone $\sigma\subset N$ lies in $\defcolor{M}:=M_\ZZ\otimes_\ZZ\RR$,
and is
\[
   \defcolor{\sigma^\vee}\ :=\ \{ u\in M\mid  u\cdot v  \geq 0\,,\ \forall v\in\sigma\}\,.
\]
This is again a (polyhedral) cone in $M$ and $(\sigma^\vee)^\vee=\sigma$.
A \demph{face} of a cone $\sigma$ is a subset of the form $\{v\in\sigma\mid  u\cdot v = 0\}$, for some
$u\in\sigma^\vee$. 
A face of a cone is another cone.
The \demph{relative interior $\sigma^\circ$}  of a cone $\sigma$ is the complement in $\sigma$ of its proper faces.

The minimal face of a cone $\sigma$ is a linear space \defcolor{$L$}, called its \demph{lineality space}.
This is the maximal linear subspace contained in $\sigma$.
The dual cone to $L$ is its annihilator, $\defcolor{L^\perp}=L^\vee$ and $\sigma^\vee\subset L^\perp$.
The lineality space of the dual cone $\sigma^\vee$ is the annihilator $\sigma^\perp$ of $\sigma$, and the dual to
$\sigma^\perp$ is the linear span \demph{$\langle\sigma\rangle$} of $\sigma$, which is also the annihilator of
$\sigma^\perp$. 

We record one technical fact about faces of cones and their duals.
It is a consequence of Equation (11) on page 13 in~\cite{Fulton}, and its proof is nearly the same as that of its discrete version, which is
Proposition 2 in {\it  loc.~cit.}, also Proposition 1.3.16 in~\cite{CLS}.

\begin{proposition}\label{P:coneFact}
 Let $\sigma,\tau\subset N$ be cones with $\tau$ a face of $\sigma$.
 Then for any $w\in\tau^\vee$, there are $u,\ell\in\sigma^\vee$ with $\ell\in\tau^\perp$ such that 
 $w=u-\ell$.
\end{proposition}

A \demph{fan} $\Sigma\subset N$ is a finite collection of polyhedral cones in $N$ with the property that every face of
every cone $\sigma$ in $\Sigma$ is again a cone in $\Sigma$, and if $\sigma,\tau$ are cones in $\Sigma$, then
$\sigma\cap\tau$ is a face common to both $\sigma$ and $\tau$.
The cones in a fan have a common lineality space.
A fan $\Sigma$ is \demph{complete} if every point of $N$ lies in some cone of $\Sigma$.
If $v_1,\dotsc,v_k\in N_\ZZ$, then $\cone\{v_1,\dotsc, v_k\}\subset N$ is \demph{rational}.
The dual of a rational cone is again rational as are all of its faces.
A fan $\Sigma\subset N$ is rational when each of its cones is rational.

Given a rational cone $\sigma\subset N$, let $\defcolor{S_\sigma}:=\sigma^\vee\cap M_\ZZ$ be the set of characters of
$\TT_N$ that lie in $\sigma^\vee$.
This is a finitely generated submonoid of $M_\ZZ$ that is saturated (whenever $u\in M_\ZZ$ with 
$m u\in S_\sigma$ for some $m\in\NN$, then $u\in S_\sigma$).
We define $\defcolor{W_\sigma}:=\spec \ZZ[S_\sigma]$, which is a normal affine scheme.
(Normal as $S_\sigma$ is saturated.)
When $\tau\subset\sigma$ is a face of $\sigma$, we have $\sigma^\vee\subset\tau^\vee$.
Then $S_\sigma\subset S_\tau$ and the induced map $W_\tau\hookrightarrow W_\sigma$ is an open inclusion.
Suppose that $\Sigma$ is a rational fan in $N$.
The \demph{toric scheme $X_\Sigma$} is obtained by gluing the affine schemes $W_\sigma$ for $\sigma$ a cone in $\Sigma$
along common subschemes corresponding to smaller cones in $\Sigma$,
\[
   X_\Sigma\ :=\ \bigcup_{\sigma\in\Sigma} W_\sigma\,.
\]
The scheme $X_\Sigma$ is normal as each $W_\sigma$ is normal and normality is a local property.

The map $\Delta\colon S_\sigma\to M_\ZZ\times S_\sigma$ given by $\Delta(u)=(u,u)$ induces a ring map
$\Delta\colon \ZZ[S_\sigma]\to \ZZ[M_\ZZ]\otimes \ZZ[S_\sigma]$, so that $\ZZ[S_\sigma]$ is a $\ZZ[M_\ZZ]$-comodule.
This induces an action $\TT_N\times W_\sigma\to W_\sigma$ of the group scheme $\TT_N$ on $W_\sigma$.
This action is compatible with the inclusions $W_\tau\hookrightarrow W_\sigma$ when $\tau$ is a face of $\sigma$, and
therefore gives an action of $\TT_N$ on the toric scheme $X_\Sigma$.

If $L$ is the lineality space of $\Sigma$, then $S_L=L^\perp\cap M_\ZZ$, which is a free abelian summand of $M_\ZZ$.
Then $W_L=\spec\ZZ[S_L]$ is the quotient $\TT_N/\TT_L$, where $\TT_L$ is the group subscheme of $\TT_N$
generated by the cocharacters in $L\cap N_\ZZ$.
The action of $\TT_N$ on $X_\Sigma$ factors through the quotient $\TT_N/\TT_L=W_L$, so that $W_L$ is isomorphic to a dense
orbit of this action.
Thus $X_\Sigma$ is a normal scheme equipped with an action of a torus $\TT_N$ having a dense orbit.

Given any field $K$, $X_\Sigma$ has a set $X_\Sigma(K)$ of $K$-rational points constructed by the same gluing procedure from
the points $W_\sigma(K)$ of $\spec K[S_\sigma]$ with residue field $K$.
There is another perspective that associates a set $X_\Sigma(\calM)$ to a commutative monoid
$(\calM,*,1_\calM)$.
This is functorial in that a map $f\colon \calM\to\calM'$ of monoids induces a map 
$f_*\colon X_\Sigma(\calM)\to X_\Sigma(\calM')$. 

The construction of $X_\Sigma(\calM)$ is similar to that of the scheme $X_\Sigma$.
For each cone $\sigma$ of $\Sigma$, set $\defcolor{W_\sigma(\calM)}:=\defcolor{\Hom_{\mon}(S_\sigma,\calM)}$, the set
of monoid homomorphisms, which are maps  $\varphi\colon S_\sigma\to\calM$ with $\varphi(0)=1_{\calM}$ and
$\varphi(a+b)=\varphi(a)*\varphi(b)$ for all $a,b\in S_\sigma$. 
Restriction gives inclusion maps $W_\tau(\calM)\hookrightarrow W_\sigma(\calM)$, with $X_\Sigma(\calM)$
constructed by gluing as before. 

Multiplication in a field $K$ gives it the structure of a monoid with $0$ an absorbing element ($0x=0$ for all $x\in K$).
The restriction of any algebra homomorphism $K[S_\sigma]\to K$ to $S_\sigma$ is a monoid homomorphism, and every monoid
homomorphism $\varphi\colon S_\sigma\to K$ extends by linearity to an algebra homomorphism $K[S_\sigma]\to K$.
Thus the two definitions for $W_\sigma(K)$ agree, and so we may construct $X_\Sigma(K)$
by gluing sets of monoid homomorphisms.

This illuminates some structure of the set $X_\Sigma(K)$ of $K$-points.
If $t\in \TT_N(K)$ and $\varphi\in\Hom_{\mon}(S_\sigma,K)$, then $t.\varphi$ is the monoid homomorphism such that
for $u\in S_\sigma$,
\[
   (t.\varphi)(u)\ =\ t^u \varphi(u)\,,
\]
where \defcolor{$t^u$} is the value of the character $u$ at $t$.

The maps of monoids $\RR_{\geq}\hookrightarrow\RR\hookrightarrow\CC\twoheadrightarrow \RR_{\geq}$ with the last map
$z\mapsto|z|$ has composition the identity.
For any rational fan $\Sigma$, this induces maps
\[
  X_{\Sigma}(\RR_{\geq})\ \lhra\ X_{\Sigma}(\RR)\ \lhra\ 
   X_{\Sigma}(\CC)\ \relbar\joinrel\twoheadrightarrow\ X_{\Sigma}(\RR_{\geq})\,,
\]
whose composition is the identity.
The set $X_{\Sigma}(\RR_{\geq})$ is the \demph{nonnegative part} of the toric variety $X_\Sigma$ and the map
$X_\Sigma(\CC)\twoheadrightarrow X_\Sigma(\RR_{\geq})$ is a version of the moment map~\cite{vilnius}.

\section{Irrational Affine Toric Varieties}
\label{S:Affine}

Irrational affine toric varieties are treated in the papers~\cite{CGS,PSV}.
We update that treatment, with an eye towards the classical development in~\cite[Ch.~4]{GBCP} and~\cite[Ch~1]{CLS}.

Let $M$ and $N$ be dual finite-dimensional real vector spaces, and write the pairing $M\times N\to\RR$ as
$(u,v)\mapsto u\cdot v$.
The vector space $N$ is the torus for our irrational toric varieties.
We often write $N$ as \defcolor{$T_N$} and use multiplicative
notation for its group operation.
For $v\in N$, we define the continuous homomorphism $\gamma_v\colon M\to \RR_>$ by
 \begin{equation}\label{Eq:gammav}
  \defcolor{\gamma_v}\ \colon\ M\ni u\ \longmapsto\ \exp(-u\cdot v)\,.
 \end{equation}
(The negative sign is for compatibility with certain limits.)
This map $v\mapsto\gamma_v$ is a group isomorphism $T_N\xrightarrow{\,\sim\,} \Hom_c(M,\RR_>)$, where \defcolor{$\Hom_c(M,\RR_>)$} is the
multiplicative group of continuous homomorphisms from $M$ to the multiplicative group $\RR_>$.
Then $v\mapsto\gamma_v$ identifies $T_N$ with $\Hom_c(M,\RR_>)$.
When $t=\gamma_v$, we write $\defcolor{t^u}$ for $\gamma_v(u)$.
Elements of $M$ are characters (continuous multiplicative homomorphisms to $\RR_>$) of $T_N$ and elements $t\in T_N$ are
\demph{cocharacters}. 
For a linear subspace $L\subset N$, we will write \defcolor{$T_L$} for the corresponding subgroup of $T_N$.

Let $\calA$ be a finite subset of $M$ and set $\sigma:=\cone(\calA)$, a polyhedral cone in $M$.
A subset $\calF\subset\calA$ is a \demph{face} of $\calA$ if it consists of the points of $\calA$ lying on a face $\tau$
of $\sigma$.  
Necessarily, $\cone(\calF)=\tau$.
Let $\defcolor{\varphi_\calA}\colon T_N\to \RR^\calA_{\geq}$ be the map given by
 \[
   T_N\ \ni\ t\ \longmapsto\ ( t^a \mid a\in\calA) \ \in\ \RR^\calA_{\geq}\,.
 \]
We are using $\calA$ as an index set, so that $\defcolor{\RR^\calA_{\geq}}=\RR^{|\calA|}_{\geq}$ is the set of
$|\calA|$-tuples of nonnegative real numbers whose coordinates are indexed by elements of $\calA$.
(This is the nonnegative orthant of $\RR^\calA$.)
The map $\varphi_\calA$ is a group homomorphism from $T_N$ to \defcolor{$\RR^\calA_>$},
which is the set of points of $\RR^\calA_{\geq}$ with nonzero coordinates.
Write \defcolor{$Y_\calA^\circ$} for the image of $T_N$ under $\varphi_\calA$, 
and let \defcolor{$Y_\calA$} be the closure of $Y_\calA^\circ$ in the usual topology on
$\RR^\calA_{\geq}$. 
We call $Y_\calA$ an \demph{irrational affine toric variety}.
It inherits a continuous $T_N$-action from the homomorphism $\varphi_\calA$.

\begin{remark}
 Under the map $\gamma\colon N\xrightarrow{\,\sim\,}T_N$ and the coordinatewise map 
 $-\log\colon\RR^\calA_>\xrightarrow{\sim}\RR^\calA$, the map
 $\varphi$ becomes the linear map $N\to\RR^\calA$,
\[
   N\ \ni\ v\ \longmapsto\ (a\cdot v \mid a\in A)\ \in\ \RR^A\,.
\]
 This is why irrational toric varieties in algebraic statistics are called log-linear models.\hfill$\diamond$
\end{remark}

The kernel of $\varphi_\calA$ is $T_{\calA^\perp}$, where $\calA^\perp$ is the subspace of $N$ that annihilates $\calA$,
\[
   \defcolor{\calA^\perp}\ :=\ \{ v\in N\,\mid\, a\cdot v=0\quad\forall a\in\calA\}\,.
\]
and thus $Y^\circ_\calA$ is homeomorphic to $T_N/T_{\calA^\perp}\  (\simeq N/\calA^\perp$).
When $\calA\subset M_\ZZ$, the ideal of $Y_\calA$ is spanned by binomials determined by $\calA$~\cite[Lem.~4.1]{GBCP}. 
This remains true when $\calA\subset M$.

\begin{proposition}\label{P:IATV}
 The irrational affine toric variety $Y_\calA$ is the set of points $z\in\RR^\calA_\geq$ that satisfy all binomial
 equations of the form 
 \begin{equation}\label{Eq:binomials}
   \prod_{a\in\calA} z_a^{\lambda_a}\ =\ \prod_{a\in\calA} z_a ^{\mu_a}\ ,
 \end{equation}
 where $\lambda,\mu\in\RR^\calA_\geq$ satisfy $\sum_{a\in\calA} \lambda_a a=\sum_{a\in\calA} \mu_a a$.
\end{proposition}

This is stated without proof as Equation (8) after Theorem 2.2 in~\cite{PSV}.
If $\calA$ lies on an affine hyperplane in $M$, this was shown in~\cite[Prop.~B.3]{CGS} and that proof is easily modified
for the general case, see~\cite[Prop.~5.4]{Ata_Thesis} for details.

For each face $\calF\subset\calA$, there is an inclusion $\RR^\calF_\geq\subset\RR^\calA_\geq$, where
$\RR^\calF_\geq$ is the set of $z\in\RR^\calA_\geq$ whose coordinates $z_a$ are zero for $a\not\in\calF$.
Then $Y^\circ_\calF\subset\RR^\calF_\geq$ is the image of $T_N$ under the map $\varphi_\calF$, which is also the composition
of $\varphi_\calA$ with the projection to the coordinate orthant $\RR^\calF_\geq$.
The proof of Proposition~\ref{P:IATV} shows that $Y^\circ_\calF\subset Y_\calA$ and also that 
 \begin{equation}\label{Eq:decomposition}
    Y_\calA \ =\ \bigsqcup Y^\circ_\calF\,,
 \end{equation}
the (disjoint) union over faces $\calF$ of $\calA$.
This is also the decomposition of $Y_\calA$ into orbits of $T_N$,
where the orbit $Y^\circ_\calF$ is in bijection with $T_N/T_{\calF^\perp} \simeq N/\calF^\perp$.

The affine toric variety $Y_\calA$ is a subset of the nonnegative orthant $\RR^\calA_\geq$.
The \demph{tautological map} \defcolor{$\pi_\calA$} parameterizing $\cone(\calA)$ has domain $\RR^\calA_\geq$:
\[
  \pi_\calA\ \colon\    \RR^\calA_\geq\ \ni\ \lambda=(\lambda_a\mid a\in\calA)\ 
   \longmapsto\ \sum_{a\in\calA} \lambda_a a\ \in\ \cone(\calA)\,.
\]
Consequently, $\pi_\calA(Y_\calA)\subset\cone(\calA)$.
By Birch's Theorem~\cite{Birch}, this map is a homeomorphism between $Y_\calA$ and $\cone(\calA)$.

\begin{proposition}[Birch]\label{P:Birch}
 The restriction of the tautological map $\pi_\calA\colon  \RR^\calA_\geq\to\cone(\calA)$ to the irrational affine toric
 variety $Y_\calA$ is a homeomorphism,
\[ 
   \pi_\calA\ \colon\ Y_\calA\ \xrightarrow{\ \sim\ }\ \cone(\calA)\,.
\]
\end{proposition}

If $\calA$ lies on an affine hyperplane in $M$, this was shown in~\cite[Thm.~1.10]{ASCB}, and its proof can be modified
for the general case, see~\cite{Ata_Thesis} for details.
For another, independent proof, see~\cite[Thm.~26]{Krasauskas}.
By~\eqref{Eq:decomposition}, each orbit $Y^\circ_\calF$ is mapped
homeomorphically to the relative interior of $\cone(\calF)$. 

The results together show that $Y_\calA$ is a $T_N$-equivariant cell complex, with one cell for each face $F$ of
$\cone(\calA)$, were the $T_N$-action  realizes that  cell as $T_N/T_{F^\perp}\ (\simeq N/F^\perp)$.

\section{Irrational Toric Varieties From Fans}
\label{S:irrational}

While irrational affine toric varieties were developed before, the construction we give here of irrational toric varieties from irrational
fans is novel. 
We construct a $T_N$-equivariant cell complex $Y_\Sigma$ associated to a fan $\Sigma\subset N$, which we call an irrational
toric variety.
This follows the construction of the nonnegative part $X_\Sigma(\RR_\geq)$ of a toric variety at the end of
Section~\ref{S:classical}, and it yields this nonnegative part when $\Sigma$ is a rational fan.
For each cone $\sigma$ of $\Sigma$, we construct a topological space $V_\sigma$ with a $T_N$-action that is equivariantly
homeomorphic to an irrational affine toric variety $Y_\calA$, but not canonically.  
Given a face $\tau\subset\sigma$, there is a natural equivariant inclusion $V_\tau\hookrightarrow V_\sigma$.
The irrational toric variety $Y_\Sigma$ is constructed by gluing the irrational affine toric varieties $V_\sigma$ for
$\sigma$ a cone in $\Sigma$ along common subvarieties corresponding to smaller cones in $\Sigma$.

\subsection{Irrational affine toric varieties from cones}\label{SS:Affine_cone}
Let $C\subset M$ be a polyhedral cone.
Define \defcolor{$\Hom_c(C,\RR_\geq)$} to be the set of all monoid homomorphisms $\varphi\colon C\to\RR_\geq$ that are
continuous on the relative interior of each face of $C$.
We endow this with the weakest topology such that every evaluation at a point of $C$ is continuous.
That is, a sequence $\{\varphi_n\mid n\in\NN\}\subset\Hom_c(C,\RR_\geq)$ converges to
$\varphi\in\Hom_c(C,\RR_\geq)$ if and only if, for every $u\in C$, the sequence 
$\{\varphi_n(u)\mid n\in\NN\}$ of real numbers  converges to $\varphi(u)$.

\begin{example}\label{Ex:ray}
 Suppose that $M=\RR$ and $C=[0,\infty)$.
 Let $\varphi\in\Hom_c(C,\RR_\geq)$.
 Because $\varphi$ is a monoid homomorphism, $\varphi(0)=1$.
 Set $\defcolor{\alpha}:=\varphi(1)\geq 0$.
 As $\varphi$ is a monoid homomorphism, for every $n\in\NN$ with $n\geq 1$, we have $\varphi(n)=\alpha^n$.
 For $n\in\NN$, we have  $\varphi(1/n)=\alpha^{1/n}$, as 
\[
   \alpha\ =\ \varphi(1)\ =\ \varphi(n\cdot\tfrac{1}{n})\ =\ (\varphi(1/n))^n\,.
\]
 Similarly, if $r\in\QQ$ is positive, then $\varphi(r)=\alpha^r$.
 By the continuity of $\varphi$ on the interior $(0,\infty)$ of $C$, $\varphi(s)=\alpha^s$ for $s>0$.
 Here, if $\alpha=0$, then $\alpha^s=0$ and if $\alpha>0$, then $\alpha^s=\exp(s\log\alpha)$.
 If we further set $\alpha^0:=1$ for any $\alpha\geq 0$ (that is, $0^0=1$), then $\varphi(s)=\alpha^s$ for $s\geq 0$.

 For $\alpha\in[0,\infty)$ write \defcolor{$\varphi_\alpha$} for the monoid homomorphism such that 
 $\varphi_\alpha(s)=\alpha^s$. 
 Figure~\ref{F:graphs} shows the graphs of $\varphi_\alpha$ for several values of $\alpha$.
\begin{figure}[htb]
  \centering 
  \begin{picture}(232,89)(-18,-10)
    \put(-0.5,-0.5){\includegraphics{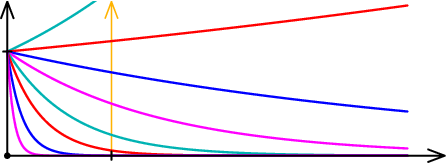}}
    \put(-8,49){$1$} \put(50,-10){$1$}  \put(57,65){$\alpha$}  \put(190,-10){$C$} 
    \put(-18,67){$\RR_\geq$}
  \end{picture}
  \caption{Graphs of $\varphi_\alpha\colon[0,\infty)\to\RR_\geq$ for several values of $\alpha\in[0,\infty)$.}
  \label{F:graphs}
\end{figure}
 For $s>0$, we have $\lim_{\alpha\to 0}\varphi_\alpha(s)=0$ while $\varphi_\alpha(0)=1$ for all $\alpha$, so
 that $\lim_{\alpha\to 0}\varphi_\alpha=\varphi_0$.
 Thus the evaluation map $\varphi\mapsto\varphi(1)$ induces a
 homeomorphism between $\Hom_c(C,\RR_\geq)$ and $[0,\infty)$.
 Restricting a map $\varphi\in \Hom_c(C,\RR_\geq)$ to $\NN\subset C$ gives a homeomorphism 
 between $\Hom_c(C,\RR_\geq)$ and $\Hom_{\mon}(\NN,\RR_\geq)$.\hfill$\diamond$
\end{example}

\begin{lemma}\label{L:HomFace}
 Let $C\subset M$ be a polyhedral cone.
 For any $\varphi\in\Hom_c(C,\RR_\geq)$, the set $\{u\in C\mid\varphi(u)>0\}$ is a face of $C$.
\end{lemma}

We will call this face the \demph{support} of $\varphi$ and write \defcolor{$\supp(\varphi)$} for this face.

\begin{proof}
 As in Example~\ref{Ex:ray}, if $s\in\RR_\geq$ and $u\in C$, then
 $\varphi(su)=(\varphi(u))^s$. 
 Suppose that $u$ lies in the relative interior \defcolor{$\tau^\circ$} of a face $\tau$ of $\sigma$ and $\varphi(u)=0$.
 Then $\varphi(su)=0$ for $s>0$.
 For $w\in\tau^\circ$, there is an $s>0$ with 
 $w-su\in\tau^\circ$, and so $\varphi(w)=\varphi(w-su)\varphi(su)=0$.

 Let $S:=\{u\in C\mid\varphi(u)>0\}$, which is closed under sum and multiplication by $\RR_\geq$, so it is a convex
 cone. 
 If $S$ meets the relative interior  $\tau^\circ$ of a face $\tau$ of $C$, then it contains 
 $\tau^\circ$, by the previous arguments.
 In fact, $\tau\subset S$.
 Indeed, if $u\in\tau$ and $w\in\tau^\circ$, then $w+u\in\tau^\circ$ so that
\[
   0\ \neq\ \varphi(w+u)\ =\ \varphi(w)\varphi(u)\,,
\]
 so that $\varphi(u)\neq 0$ and thus $u\in S$.
 A convex union of faces of a polyhedral cone is a face of that cone, which completes the proof.
\end{proof}

\begin{example}\label{Ex:LinSpace}
 If $L$ is a linear subspace of $M$, then for any $\varphi\in\Hom_c(L,\RR_\geq)$ and $x\in L$, we have that 
 $\varphi(x)>0$.
 Indeed, $-x\in L$ so $1=\varphi(0)=\varphi(x+(-x))=\varphi(x)\varphi(-x)$.
 In particular, $\Hom_c(L,\RR_\geq)=\Hom_c(L,\RR_>)$, which is homeomorphic to 
 $T_{L^\vee}=T_N/T_{L^\perp}\simeq N/{L^\perp}$.\hfill$\diamond$
\end{example}

Applying $\Hom_c(-,\RR_>)$ to the map $C\to M\oplus C$ given by $u\mapsto (u,u)$ induces a map
 \begin{equation}\label{Eq:T-action}
  \defcolor{\mu}\ \colon\ \Hom_c(M,\RR_>)\ \times\ \Hom_c(C,\RR_\geq)\ \longmapsto\ \Hom_c(C,\RR_\geq)\,,
 \end{equation}
written $\mu(t,\varphi)= t.\varphi$, which is the map defined at $u\in C$ by $t.\varphi(u)=t(u)\varphi(u)=t^u\varphi(u)$.
This gives a continuous action of $T_N=\Hom_c(M,\RR_>)$ on $\Hom_c(C,\RR_\geq)$.

\begin{lemma}\label{L:HomIsAffine}
 Let $\calA\subset\calM$ be finite and set $C:=\cone(\calA)$.
 Then the map $f_\calA\colon\Hom_c(C,\RR_\geq)\mapsto \RR^\calA_\geq$ given by
\[
   f_\calA\ \colon\ \varphi\ \longmapsto\ (\varphi(a)\,\mid\, a\in\calA)\,,
\]
 is a $T_N$-equivariant homeomorphism between $\Hom_c(C,\RR_\geq)$ and the irrational affine toric variety $Y_\calA$.
 In particular, $\Hom_c(C,\RR_\geq)$ is homeomorphic to $C$ under the map 
 $\varphi\mapsto \sum_{a\in\calA}\varphi(a) a$.
\end{lemma}

Observe that this homeomorphism depends on the choice of a generating set $\calA$ for $C$.

\begin{proof}
 Let $\varphi\in\Hom_c(C,\RR_\geq)$ and suppose that $u\in C$ has two representations as a nonnegative 
 combination of elements of $\calA$, $u=\sum_a\lambda_a a=\sum_a \mu_a a$ for $\lambda,\mu\in\RR^\calA_\geq$.
 Since
 \begin{equation}\label{Eq:HomIsAffine}
   \varphi(u)\ =\ \varphi\Bigl(\sum \lambda_a a\Bigr)\ =\ \prod_{a\in\calA} (\varphi(a))^{\lambda_a}\ ,
 \end{equation}
 $f_\calA(\varphi)$ satisfies the equations~\eqref{Eq:binomials}, and therefore $f_\calA(\varphi)$ lies in 
 $Y_\calA$, by Proposition~\ref{P:IATV}.

 Conversely, let $z\in Y_\calA$.
 Then by~\eqref{Eq:HomIsAffine} and the equations~\eqref{Eq:binomials} for $Y_\calA$, the function
 $\defcolor{\varphi_z}\colon\calA\to\RR_\geq$ defined by 
 $\varphi_z(a)=z_a$ extends to a monoid homomorphism $C=\cone \calA \to\RR_\geq$.
 By the decomposition~\eqref{Eq:decomposition}, there is a face $\calF$ of $\calA$ such that $z\in Y^\circ_\calF$.
 Then $\varphi_z$ is continuous and nonvanishing on $\defcolor{F}:=\cone(\calF)$.
 We show that  $\varphi_z$ is identically zero on $C\smallsetminus F$, and thus it lies in $\Hom_c(C,\RR_\geq)$.

 Note that if $b\in\calA{\smallsetminus}\calF$, then $z_b=0$ and thus $\varphi_z(b)=0$.
 If $u\in C{\smallsetminus} F$, then in any expression  $u=\sum_{a\in\calA} \lambda_a a$ with $\lambda_a\geq 0$, 
 there is some $b\in\calA{\smallsetminus}\calF$ with $\lambda_b>0$.
 Thus
\[
   \varphi_z(u)\ =\ \bigl(\varphi_z(b)\bigr)^{\lambda_b}\, \varphi_z 
       \Bigl(\, \sum_{a\in\calA\smallsetminus\{b\}}\lambda_a a\Bigr)\ =\ 0\,.
\]

 The maps $\varphi\mapsto f_\calA(\varphi)$ and $z\mapsto \varphi_z$ are inverse bijections
 which are continuous and therefore homeomorphisms.
 The given formulas show that they are $T_N$-equivariant.
 The identification between $\Hom_c(C,\RR_\geq)$ and $C$ now follows by Birch's Theorem (Proposition~\ref{P:Birch}).
\end{proof}

\begin{example}\label{Ex:ConeEx}
 Let $C:=\cone\{a,b\}$, where $\defcolor{a}=(-\sqrt{2},1)$ and $\defcolor{b}=(1,0)$.
 Elements of $C$ have unique expressions as nonnegative combinations of $a$ and $b$, so that a map
 $\varphi\in\Hom_c(C,\RR_\geq)$ is determined by its values $\varphi(a)$ and $\varphi(b)$, which may be any two nonnegative
 numbers.
 Thus the map $\defcolor{\psi}\colon\varphi\mapsto \varphi(a) a + \varphi(b)b$ is a homeomorphism between 
 $\Hom_c(C,\RR_\geq)$  and $C$.

 Adding a generator $\defcolor{c}=(1,1)$ of $C$, if $\varphi\in\Hom_c(C,\RR_\geq)$, 
 then $\varphi(c)=\varphi(a)\varphi(b)^{1+\sqrt{2}}$, as we have $c = a+(1{+}\sqrt{2})b$.
 By Lemma~\ref{L:HomIsAffine}, $\varphi\mapsto(\varphi(a),\varphi(b),\varphi(c))$ is a homeomorphism
 $\Hom_c(C,\RR_\geq)\xrightarrow{\sim}Y_{\{a,b,c\}}$.
 Composing with the tautological map $\pi_{\{a,b,c\}}$ gives 
\[
   \varphi\ \longmapsto\  \varphi(a) a + \varphi(b)b + \varphi(c)c\ =\  
    \varphi(a) a + \varphi(b)b + \varphi(a)\varphi(b)^{1+\sqrt{2}} c\,,
\]
 which is a  homeomorphism between  $\Hom_c(C,\RR_\geq)$  and $C$ that is different from $\psi$.
 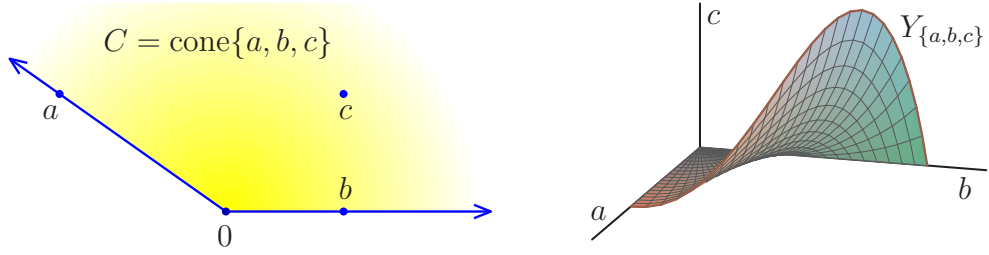
\begin{figure}[htb]
 \[
   \begin{picture}(187,90)(0,-6)
    \put(0,0){\includegraphics[height=84pt]{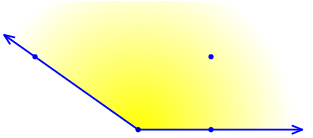}}
    \put(15,40){$a$}    \put(127,10){$b$}    \put(127,40){$c$}   \put(81,-8){$0$}
    \put(38,66){$C=\cone\{a,b,c\}$}
   \end{picture}   
     \qquad\quad
   \begin{picture}(149,90)
    \put(0,0){\includegraphics[height=90pt]{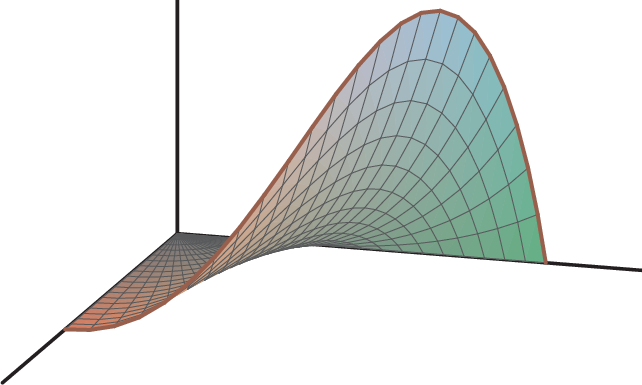}}
    \put(0,8){$a$}    \put(139,16){$b$}    \put(44,82){$c$}   
    \put(117,77){$Y_{\{a,b,c\}}$}
   \end{picture}
 \]
 \caption{Cone and irrational affine toric variety from Example~\ref{Ex:ConeEx}.}
 \label{F:IATV}
\end{figure}
 This illustrates that the homeomorphism depends upon the choice of generators.
 Figure~\ref{F:IATV} shows the cone $C$ and the irrational affine toric variety $Y_{\{a,b,c\}}$.\hfill$\diamond$
\end{example}

In Lemma~\ref{L:HomFace} and in the proof of Lemma~\ref{L:HomIsAffine}, faces $F$ of the cone $C$
correspond to monoid homomorphisms $\varphi$ that vanish on $C{\smallsetminus}F$ and are nonzero on $F$.
For a face $F$ of $C$, let \defcolor{$\langle F\rangle$} be its linear span.
Recall that $\Hom_c(\langle F\rangle,\RR_\geq)=\Hom_c(\langle F\rangle, \RR_>)$, which  is a single 
$T_N$-orbit isomorphic to $T_N/T_{F^\perp}$, where $\defcolor{F^\perp}\subset N$ is the annihilator of
$F$. 
Write $\defcolor{\varepsilon_F}\in\Hom_c(\langle F\rangle, \RR_>)$ for the constant homomorphism, $\varepsilon_F(u)=1$ for
all $u\in\langle F\rangle$.
Then $\Hom_c(\langle F\rangle, \RR_>)=T_N.\varepsilon_F$.
Restriction to $F$ followed by extension by 0 to the rest of $C$ gives a $T_N$-equivariant
map $\Hom_c(\langle F\rangle, \RR_>) \to \Hom_c(C,\RR_\geq)$ which sends the constant map $\varepsilon_F$ to the element 
of $\Hom_c(C,\RR_\geq)$ (still written $\varepsilon_F$) whose value at $u\in C$ is
\[
    \varepsilon_F(u)\ =\ \left\{
     \begin{array}{rcl} 1&\ & u\in F\\ 0&&u\in C\smallsetminus F\end{array}\right.\ .
\]
Set $\defcolor{\calO_F}:=T_N.\varepsilon_F\subset\Hom_c(C,\RR_\geq)$, the orbit through $\varepsilon_F$, which is
the image of $\Hom_c(\langle F\rangle, \RR_>)$.

\begin{corollary}
  For any face $F$ of $C$, $\calO_F$ consists of those homomorphisms in $\Hom_c(C,\RR_\geq)$ that vanish on 
  $C{\smallsetminus}F$ and are nonzero on $F$.
  The map  $\Hom_c(\langle F\rangle, \RR_>) \to \Hom_c(C,\RR_\geq)$ is an inclusion with image $\calO_F$, and 
  we have the orbit decomposition
 \begin{equation}\label{Eq:orbit_decomposition}
  \Hom_c(C,\RR_\geq)\ =\ \bigsqcup_{\mbox{\scriptsize  $F$ a face of $C$}} \calO_F\ .
 \end{equation}
\end{corollary}

\begin{proof}
 By definition of the map $\Hom_c(\langle F\rangle, \RR_>) \to \Hom_c(C,\RR_\geq)$, its image consists of homomorphisms that
 vanish on $C{\smallsetminus} F$ and are nonzero on $F$.
 Suppose that $\varphi\in\Hom_c(C,\RR_\geq)$ vanishes on $C{\smallsetminus} F$ and is nonzero on $F$.
 If we restrict $\varphi$ to $F$, it is a monoid homomorphism from $F$ to the group $\RR_>$ and therefore has a unique
 extension to $\langle F\rangle$.
 Thus $\varphi$ lies in the image of the map $\Hom_c(\langle F\rangle, \RR_>) \to \Hom_c(C,\RR_\geq)$, proving
 the first assertion.
 The map is an injection, as an element of $\Hom_c(\langle F\rangle, \RR_>)$ is determined by its restriction to $F$.

 Finally, the decomposition~\eqref{Eq:orbit_decomposition} follows
 from~\eqref{Eq:decomposition} and Lemma~\ref{L:HomIsAffine}.
\end{proof}

Let $v$ be an element of $C^\vee\subset N$ that exposes the face $F$ of $C$, that is, $v^\perp\cap C=F$ and for 
$u\in C{\smallsetminus} F$, $u\cdot v>0$.
Recall the multiplicative homomorphism $\gamma_v$ for $v\in N$ defined in~\eqref{Eq:gammav} and the resulting identification of $T_N$ with
$\Hom_c(M,\RR_>)$. 
For $s\in \RR$, we have the element $\gamma_{sv}\in T_N$ whose value at $u\in M$ is $\gamma_{sv}(u)=\exp(-su\cdot v)$.
The map $\RR\to T_N$ defined by $s\mapsto \gamma_{sv}$ is a \demph{one-parameter subgroup} of $T_N$.
Recall that $\varepsilon_C$ is the constant map on $C$, taking the value $1$ at every point $u\in C$.

\begin{lemma}\label{L:limitOfDistinguishedPoints}
 With these definitions, we have
 \begin{enumerate}
  \item  $\varepsilon_F=\lim_{s\to \infty} \gamma_{sv}.\varepsilon_C$.
  \item  If $F$ is a face of $E$ and both are faces of $C$, then 
          $\varepsilon_F=\lim_{s\to \infty} \gamma_{sv}.\varepsilon_E$.
 \end{enumerate}
\end{lemma}

\begin{proof}
 If $\gamma_{sv}.\varepsilon_C$ has a limit as $s\to\infty$ in $\Hom_c(C,\RR_\geq)$, then its value at $u\in C$ is 
\[
   \lim_{s\to\infty} (\gamma_{sv}.\varepsilon_C) (u)\ =\ 
   \lim_{s\to\infty} \gamma_{sv}(u) \varepsilon_C(u)\ =\ 
   \lim_{s\to\infty} \exp(-s\, u\cdot v)\ =\ 
   \left\{\begin{array}{rcl} 0&\ & u\not\in F\\ 1&& u\in F\end{array}\right.\ ,
\]
 which is $\varepsilon_F$.
 (We are using that $\varepsilon_C(u)=1$.)
 The last equality is because as $v\in C^\vee$, $u\cdot v$ is nonnegative and the limit equals zero when $u\cdot v\neq 0$
 and it equals 1 when $u\cdot v= 0$.
 This proves (1).
 Nearly the same argument proves (2).
\end{proof}

By Lemma~\ref{L:limitOfDistinguishedPoints}(1), $\varepsilon_F\in\overline{\calO_C}$.
As $\calO_F=T_N.\varepsilon_F$, we deduce the following.

\begin{corollary}\label{Cor:orbits}
  If $F\subset E$ are faces of the cone $C$, then $\calO_F\subset\overline{\calO_E}$.
  In particular, $\calO_C$ is dense in $\Hom_c(C,\RR_\geq)$, and if $E$ is a face of $C$, then 
\[
    \overline{\calO_E}\ =\ \bigsqcup_{F\mbox{\scriptsize\ a face of }E} \calO_F\ .
\]
\end{corollary}

\subsection{Irrational toric varieties from fans}\label{SS:fans}
Let $\Sigma\subset N$ be a fan.
For a cone $\sigma\in\Sigma$, define $\defcolor{V_\sigma}:=\Hom_c(\sigma^\vee,\RR_\geq)$, the irrational affine toric variety
as in \S~\ref{SS:Affine_cone} associated to the dual cone $\sigma^\vee$ of $\sigma$.
Suppose that $\tau$ is a face of $\sigma$.
Then the inclusion $\sigma^\vee\subset\tau^\vee$ induces a map $V_\tau\to V_\sigma$ by 
restricting a monoid homomorphism on $\tau^\vee$ to $\sigma^\vee$.

\begin{lemma}\label{L:inclusion}
 The map $V_\tau\to V_\sigma$ is a $T_N$-equivariant inclusion. 
\end{lemma}

\begin{proof}
 By the definition~\eqref{Eq:T-action} of the $T_N$-action on $\Hom_c(\sigma^\vee,\RR_\geq)$, the restriction map is
 equivariant.
 It is injective because a monoid homomorphism $\varphi\in\Hom_c(\tau^\vee,\RR_\geq)$ 
 is determined by its restriction to $\sigma^\vee$.
 Indeed, let $\varphi\in\Hom_c(\tau^\vee,\RR_\geq)$  and $w\in\tau^\vee$.
 By Proposition~\ref{P:coneFact}, there are $u,\ell\in\sigma^\vee$ with $\ell\in\tau^\perp$ such that 
 $w=u-\ell$.
 Since $\tau^\perp\subset\tau^\vee$ is a linear space, $\varphi(\ell)\neq 0$  and 
 $\varphi(w)=\varphi(u)\varphi(\ell)^{-1}$.
\end{proof}

\begin{example}\label{Ex:Inclusion}
 Let $\sigma=\cone\{(1,\sqrt{2}),(0,1)\} \subset\RR^2$, and let $\tau$ be its face generated by
 $(0,1)$. 
 Then $\sigma^\vee$ is the cone $C$ of Example~\ref{Ex:ConeEx} and 
 $\tau^\vee=\{(x,y)\in\RR^2\mid y\geq 0\}=\RR\times\RR_\geq$.
 The points $a=(-\sqrt{2},1)$, $b=(1,0)$, and $c=(1,1)$  lie in both dual cones $\sigma^\vee$ and $\tau^\vee$, and
 $\tau^\vee$ has an additional generator $d:=(-1,0)$.
 Figure~\ref{F:CDATV} displays the cones $\sigma$ and $\tau$, their duals, and the associated 
\begin{figure}[htb]
\[
\begin{array}{lll}
   \begin{picture}(50,60)(-6,0)
    \put(0,0){\includegraphics{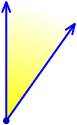}}
    \put(-6,29){$\tau$}  \put(13,43){$\sigma$}
    \put(-6,-2){$0$}
   \end{picture}
   &
   \begin{picture}(155,60)
    \put(0,0){\includegraphics{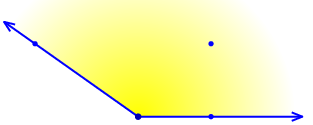}}
    \put(11,31){$a$} \put(50,43){$\sigma^\vee$} \put(93,37){$c$} \put(100,8){$b$}
    \put(33,-8){$\sigma^\perp=0$}
   \end{picture}
    &
   \begin{picture}(95,65)(-2,0)
    \put(0,0){\includegraphics[height=65pt]{figures/AffineTV}}
    \put(-2,5){$a$}  \put(21,58){$c$} \put(100,22){$b$}
    \put(58,8){$Y_{\{a,b,c\}}$}
   \end{picture}
    \\
   \begin{picture}(50,85)(-6,0)
    \put(0,0){\includegraphics{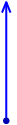}}
    \put(-6,29){$\tau$}
    \put(-6,-2){$0$}
   \end{picture}
    &
   \begin{picture}(155,85)
    \put(0,0){\includegraphics{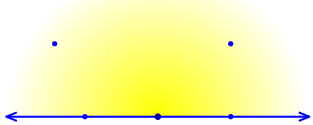}}
    \put(17,32){$a$} \put(70,39){$\tau^\vee$} \put(37,8){$d$} \put(110,8){$b$}
     \put(103,37){$c$}
    \put(-7,4){$\tau^\perp$}
    \put(73,-8){$0$}
   \end{picture}
    &
   \begin{picture}(93,85)
    \put(0,0){\includegraphics[height=67pt]{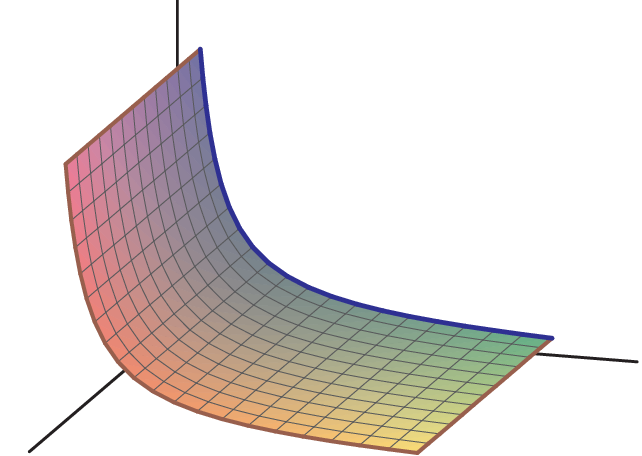}}
    \put(0,4){$a$}  \put(18,63){$d$} \put(87,16.5){$b$}
    \put(45,40){$Y_{\{a,b,d\}}$}
   \end{picture}
    
\end{array}
\]
\caption{Cones, their duals, and associated irrational affine toric varieties from Example~\ref{Ex:Inclusion}.}
\label{F:CDATV}
\end{figure}
irrational affine toric varieties $Y_{\{a,b,c\}}\simeq V_\sigma$ and  $Y_{\{a,b,d\}}\simeq V_\tau$.
The inclusion $V_\tau\hookrightarrow V_\sigma$ is induced by projecting $Y_{\{a,b,d\}}$ to 
the quadrant $\RR^{\{a,b\}}_\geq$ and then applying the inverse of the projection from $Y_{\{a,b,c\}}$.
The image of $Y_{\{a,b,d\}}$  in $Y_{\{a,b,c\}}$ only omits the $a$-axis.\hfill$\diamond$
\end{example}

The \demph{irrational toric variety} $Y_\Sigma$ associated to a fan $\Sigma\subset N$  is 
 \begin{equation}\label{Eq:ITV_Gluing}
   \defcolor{Y_\Sigma}\ :=\ \bigcup_{\sigma\in\Sigma} V_\sigma\ =\ 
    \bigcup_{\sigma\in\Sigma} \Hom_c(\sigma^\vee,\RR_\geq)\,,
 \end{equation}
the union of the irrational affine toric varieties $V_\sigma$ for $\sigma\in\Sigma$ glued together along the inclusions
$V_\tau\hookrightarrow V_\sigma$ for $\tau$ a face of $\sigma$.
For each cone $\sigma\in\Sigma$, let $\defcolor{x_\sigma}\in V_\sigma$ be the distinguished point
$\varepsilon_{\sigma^\perp}$, where $\sigma^\perp\subset\sigma^\vee$ is its lineality space.
We also let \defcolor{$U_\sigma$} be the $T_N$-orbit through $x_\sigma$, so that $U_\sigma=\calO_{\sigma^\perp}$, in the
notation of Subsection~\ref{SS:Affine_cone}. 
Note also that $U_\sigma \simeq N/\langle\sigma\rangle$.

\begin{theorem}\label{Th:ITV_structure}
 For any fan $\Sigma\subset N$, the irrational toric variety $Y_\Sigma$ is a $T_N$-equivariant cell complex.
 Each cell is an orbit and corresponds to a unique cone $\sigma\in\Sigma$.
 The cell corresponding to the cone $\sigma$ is $U_\sigma\simeq N/\langle\sigma\rangle$ and 
 $\tau\subset\sigma$ if and only if $U_\sigma\subset\overline{U_\tau}$, so that the cell structure
 of $Y_\Sigma$ and its poset of containment-in-closure is dual to that of the fan $\Sigma$.
\end{theorem}

\begin{proof}
 For each cone $\sigma\in\Sigma$, the set $V_\sigma$ is a $T_N$-equivariant cell complex whose cells are $T_N$-orbits that 
 correspond to the faces $\tau$ of $\sigma$ where the orbit $U_\tau$ corresponding to $\tau$ is identified with
 $N/\langle\tau\rangle$. 
 By Corollary~\ref{Cor:orbits}, the cell $U_\sigma$ lies in the closure of any cell $U_\tau$ for $\tau\subset\sigma$.
 Because $Y_\Sigma$ is obtained by identifying the sets $V_\sigma$ along common open subsets, these same facts hold for 
 $Y_\Sigma$.
 The last statement is a consequence of the previous statements.
\end{proof}

\begin{corollary}\label{Cor:affine_cover}
 The collection $\{ V_\sigma\mid \sigma\in\Sigma\}$ of irrational affine toric varieties  forms a $T_N$-equivariant open
 cover of $Y_\Sigma$ by irrational affine toric varieties.
\end{corollary}

\begin{proof}
 Since $Y_\Sigma$ is the union of the irrational affine toric varieties $V_\sigma$, which are $T_N$-equivariant as is the
 gluing, we only need to show that each $V_\sigma$ is open in $Y_\Sigma$. 
 By Theorem~\ref{Th:ITV_structure},
\[
   Y_\Sigma\smallsetminus V_\sigma\ =\ 
   \bigcup_{\tau\not\subset\sigma}  U_\tau\ =\ 
   \bigcup_{\tau\not\subset\sigma}  \overline{U_\tau}\ ,
\]
 as $\tau\subset\rho$ with $\tau\not\subset\sigma$ implies that $\rho\not\subset\sigma$.
 Thus $Y_\Sigma\smallsetminus V_\sigma$ is closed.
\end{proof}

For a fan $\Sigma\subset N$ and a cone $\sigma\in\Sigma$, the \demph{star} of $\sigma$ in $\Sigma$ is the fan in $N$
\[
   \defcolor{\Star(\sigma)}\ :=\ 
  \{ \langle\sigma\rangle + \tau \mid \tau\in\Sigma\mbox{ and }\sigma\mbox{ is a face of }\tau\}\,.
\]
By Theorem~\ref{Th:ITV_structure}, an orbit $U_\tau$ lies in the closure of an orbit $U_\sigma$ if and only if $\sigma$ is
a face of $\tau$.
The following corollary is a consequence of these facts and the definition of star.

\begin{corollary}\label{C:Star}
 For any cone $\sigma\in\Sigma$, the closure of the orbit $U_\sigma$ is the toric variety $Y_{\Star(\sigma)}$.
\end{corollary}

Lastly, if $\Sigma$ is rational, then $Y_\Sigma$ is the nonnegative part of the classical toric variety
$X_\Sigma$.

\begin{theorem}
 If\/ $\Sigma\subset N$ is a rational fan, then $Y_\Sigma=X_\Sigma(\RR_\geq)$.
\end{theorem}

\begin{proof}
 Because both $Y_\Sigma$ and $X_\Sigma(\RR_\geq)$ are constructed by the same gluing procedure from the sets 
 $V_\sigma=\Hom_c(\sigma^\vee,\RR_\geq)$ and $\Hom_{\mon}(S_\sigma, \RR_\geq)$ for the cones $\sigma\in\Sigma$, it
 suffices to show that these two sets are equal.
 (Recall that $S_\sigma=\sigma^\vee\cap M_\ZZ$.)
 Let $\sigma\subset M$ be a rational cone.  
 Then $\sigma^\vee$ is generated as a cone by the monoid $S_\sigma$, so that restricting a map from
 $\sigma^\vee$ to $S_\sigma$ is an injection $V_\sigma\hookrightarrow\Hom_{\mon}(S_\sigma, \RR_\geq)$.  
 Let $\calA\subset S_\sigma$ be a generating set for $\sigma^\vee$.
 The map $f_\calA$ of Lemma~\ref{L:HomIsAffine} maps both $V_\sigma$ and $\Hom_{\mon}(S_\sigma,\RR_\geq)$ to $Y_\calA$,
 with both maps isomorphisms.
 Thus the restriction map identifies $V_\sigma$ with $\Hom_{\mon}(S_\sigma,\RR_\geq)$, which completes the proof.
\end{proof}

\subsection{Maps of fans}\label{SS:maps}

Let $\Sigma\subset N$ and $\Sigma'\subset N'$ be fans in possibly different vector spaces $N$ and $N'$.
Let $Y_\Sigma$ and $Y_{\Sigma'}$ be the associated irrational toric varieties.
A map $\psi\colon Y_\Sigma\to Y_{\Sigma'}$ of irrational toric varieties is a continuous map 
$\psi\colon Y_\Sigma\to Y_{\Sigma'}$ together with a homomorphism $\Psi\colon T_N\to T_{N'}$ of topological groups
such that the following diagram commutes.
 \begin{equation} \label{Eq:Toric_Map}
  \raisebox{-27pt}{\begin{picture}(130,62)(-12,-1)
    \put(2,49){$T_{N}\times Y_{\Sigma}$} \put(53,52){\vector(1,0){44}}\put(71,56){$\mu$} \put(100,49){$Y_{\Sigma}$}
     \put(-12,27){$\Psi\times\psi$} \put(24,44){\vector(0,-1){33}}
                                   \put(105,44){\vector(0,-1){33}} \put(108,27){$\psi$}
    \put(-1,0){$T_{N'}\times Y_{\Sigma'}$} \put(53,3){\vector(1,0){44}}\put(71,7){$\mu$} \put(100,0){$Y_{\Sigma'}$}
  \end{picture}}
 \end{equation}
(The horizontal maps $\mu$ are the action, $\mu(t,y)=t.y$.)
A \demph{map of fans}, $\Psi\colon\Sigma\to\Sigma'$, is a linear map $\Psi\colon N\to N'$ such that for each cone
$\sigma\in\Sigma$, there is a cone $\sigma'\in\Sigma'$ with $\Psi(\sigma)\subset\sigma'$.

\begin{theorem}\label{Th:mapsOfFans}
 The association $\Sigma\mapsto Y_\Sigma$ is functorial for maps of fans.
 That is, if\/ $\Psi\colon \Sigma\to\Sigma'$ is a map of fans with $\Sigma\subset N$ and $\Sigma'\subset N'$, then there is a
 continuous map $\psi\colon Y_\Sigma\to Y_{\Sigma'}$ such that the diagram~\eqref{Eq:Toric_Map} commutes, where 
 the homomorphism $\Psi\colon T_N\to T_{N'}$ is induced by the linear map $\Psi\colon N\to N'$.
\end{theorem}

\begin{proof}
 Let $\Psi\colon \Sigma \to\Sigma'$ be a map of fans, where $\Sigma\subset N$ and $\Sigma'\subset N'$ are fans.
 The linear map $\Psi\colon N\to N'$ induces a homomorphism $\Psi\colon T_N\to T_{N'}$ of topological groups.
 We construct a map $\psi\colon Y_\Sigma\to Y_{\Sigma'}$ so that the diagram~\eqref{Eq:Toric_Map} commutes, by defining
 $\psi$ on each irrational affine toric variety $V_\sigma$ for $\sigma$ a cone of $\Sigma$.

 Let $\sigma\in\Sigma$ be a cone.
 Since $\psi\colon\Sigma \to\Sigma'$ is a map of fans, there is a cone $\sigma'\in\Sigma$ with
 $\Psi(\sigma)\subset\sigma'$. 
 Let $\defcolor{\Psi^*}\colon M' \to M$ be the map adjoint to $\Psi$, where $M$ and $M'$ are the vector spaces dual to $N$ and
 $N'$ respectively.
 As $\Psi(\sigma)\subset\sigma'$, we have $\Psi^*((\sigma')^\vee)\subset\sigma^\vee$.  
 Since these are polyhedral cones and $\Psi^*$ is linear, for any face $F$ of $\sigma^\vee$, its inverse image \defcolor{$F'$}
 in $(\sigma')^\vee$ is a face, and the same is true for the complement $\sigma^\vee{\smallsetminus} F$ of a face $F$ of
 $\sigma^\vee$.   

 For $\varphi\in V_\sigma=\Hom_c(\sigma^\vee,\RR_\geq)$, the composition $\defcolor{\psi(\varphi)}:=\varphi\circ\Psi^*$ is
 a monoid homomorphism  $(\sigma')^\vee \to\RR_\geq$.
 By the previous remark, the inverse image  of the support $\supp(\varphi)$ of $\varphi$ is the support of
 $\psi(\varphi)$,  and $\psi(\varphi)$ is continuous on its support.

 Thus $\psi$ maps $V_\sigma$ to $V_{\sigma'}$.
 This map is continuous as the topology is defined by point evaluation.
 It is also equivariant in the sense of~\eqref{Eq:Toric_Map}.
 Noting that it is compatible with the gluing~\eqref{Eq:ITV_Gluing} completes the proof.
\end{proof}

\section{Global Properties of Irrational Toric Varieties}
\label{S:projective}

We show that an irrational toric variety forms a monoid, that the fan $\Sigma$
is determined from limits in $Y_{\Sigma}$ and thus that $Y_\Sigma$ is a compact topological space if and only if the fan
$\Sigma$ is complete.
We also show that if $\Sigma$ is the normal fan to a polytope, then $Y_\Sigma$ has an equivariant embedding into a simplex
and is homeomorphic to that polytope.

\subsection{Irrational toric varieties as monoids}\label{SS:Monoids}
A topological monoid is a monoid that is a topological space whose operation \defcolor{$\bullet$}, called product, is
continuous. 
The affine irrational toric varieties $Y_\calA$ and $\Hom_c(C,\RR_\geq)$ are topological monoids whose structures 
are compatible with the isomorphism of Lemma~\ref{L:HomIsAffine}.
These monoids contain a dense torus which acts on them with finitely many orbits, and are thus irrational analogs of
linear algebraic monoids~\cite{Putcha,Renner}.
If we adjoin an absorbing element ${\bf 0}$ to an irrational toric variety $Y_\Sigma$, we obtain a
commutative topological monoid such that the inclusion of the irrational affine toric variety $V_\sigma$ is a monoid map,
for each cone $\sigma$ of the fan $\Sigma$.

Let $C\subset M$ be a polyhedral cone.
For $x,y\in \Hom_c(C,\RR_\geq)$, define $\defcolor{x\bullet y}\colon C\to\RR_\geq$  by  $(x\bullet y)(u)=x(u)y(u)$, for 
$u\in C$.
Let $\Phi(C)$ be the set of faces of $C$.
For faces $F,G\in \Phi(C)$, define $\defcolor{F\bullet G}:= F\cap G$.
Note that $x\mapsto \supp(x)$  is map $\Hom_c(C,\RR_\geq)\to\Phi(C)$.

\begin{proposition}\label{P:Cone_monoid}
 $\Hom_c(C,\RR_\geq)$ is a topological monoid, $\Phi(C)$ is a monoid, and $x\mapsto\supp(x)$ is a homomorphism of monoids.
 The identity of $\Hom_c(C,\RR_\geq)$ is the constant map $\varepsilon_C$ and if the lineality space $L$ of $C$ is the
 origin, then  $\Hom_c(C,\RR_\geq)$ has an absorbing element $\varepsilon_0$.
 The identity of $\Phi(C)$ is $C$ itself, and $L$ is its absorbing element.
 
 For any $u\in C$, the evaluation map $x\mapsto x(u)$ is a homomorphism of topological monoids
 $\Hom_c(C,\RR_\geq)\to\RR_\geq$. 
 For any linear map $f\colon M'\to M$ and cone $C'\subset M'$ with $f(C')\subset C$, the pullback map 
 $f^*\colon \Hom_c(C,\RR_\geq)\to \Hom_c(C',\RR_\geq)$ is a homomorphism of topological monoids.
\end{proposition}

\begin{proof}
 Let $x,y\in\Hom_c(C,\RR_\geq)$.
 As $x$ and $y$ are monoid homomorphisms, $x\bullet y$ is a monoid homomorphism.
 Also as $\supp(x\bullet y)=\supp(x)\cap\supp(y)$ is a face of $C$ and elements of $\Hom_c(C,\RR_\geq)$  are monoid
 homomorphisms that are continuous on their support, we conclude that $x\bullet y\in\Hom_c(C,\RR_\geq)$.
 This product is commutative.
 As it is defined pointwise, it is continuous, and so $\Hom_c(C,\RR_\geq)$ is a commutative 
 topological monoid.
 The other assertions are straightforward.
\end{proof}

Let $\calA\subset M$ be finite.
The orthant $\RR^\calA_\geq$ is a monoid under the Hadamard product; for $x,y\in\RR^\calA_\geq$ and $a\in\calA$,
set $(x\bullet y)_a:= x_a y_a$.
Then the injective map $f_\calA\colon\Hom_c(\cone(\calA),\RR_\geq)\to\RR^\calA_\geq$ of
Lemma~\ref{L:HomIsAffine} is a homomorphism of topological monoids whose image is $Y_\calA$.

Let $\Sigma\subset N$ be a fan.
For a cone $\sigma\in\Sigma$, the irrational affine toric variety $V_\sigma$ is a monoid under pointwise multiplication and
when $\tau\subset\sigma$ is a face, the inclusion $V_\tau\subset V_\sigma$ is a monoid homomorphism.
We define a product $\bullet$ on $\defcolor{Y_\Sigma^+}:=Y_\Sigma \sqcup \{{\bf 0}\}$, where
${\bf 0}$ is an isolated point that acts as an absorbing element.
Let $x,y\in Y_\Sigma^+$,
\begin{enumerate}
 \item
     If either $x$ or $y$ is ${\bf 0}$, then $x\bullet y={\bf 0}$.
 \item
     If there is a cone $\sigma\in\Sigma$ with $x,y\in V_\sigma$, then $x\bullet y$ is their product in $V_\sigma$.
 \item
     If there is no cone $\sigma\in\Sigma$  with $x,y\in V_\sigma$, then $x\bullet y={\bf 0}$.
      (This includes case (1).)
\end{enumerate}

Intersection of cones defines a monoid structure on $\Sigma$.
More interesting is the product on $\defcolor{\Sigma^+}:=\Sigma\sqcup\{{\bf 0}\}$, where ${\bf 0}$ is a new point that acts
as an absorbing element, in which $\sigma\bullet\tau$ is defined to be the smallest cone containing both $\sigma$ and
$\tau$ if such a cone exists, and ${\bf 0}$ otherwise.

\begin{theorem}\label{Th:ITV_monoid}
 For a fan $\Sigma\subset N$, $Y_\Sigma^+$ is a commutative topological monoid with the inclusion
 $V_\sigma\hookrightarrow Y_\Sigma\subset Y_\Sigma^+$ a homomorphism of topological monoids, for every $\sigma\in\Sigma$.
 For a map $\Psi\colon\Sigma\to\Sigma'$ of fans the functorial map $\psi\colon Y_\Sigma^+\to Y_{\Sigma'}^+$ of irrational
 toric varieties is a homomorphism of topological monoids.
 Finally, the map $Y_\Sigma^+\to\Sigma^+$ that sends an element $x$ to the cone
 $\sigma\in\Sigma$ where $x\in U_\sigma$ or to ${\bf 0}$ when $x={\bf 0}$ is a homomorphism of monoids.
\end{theorem}

This may be proven using arguments similar to those of Proposition~\ref{P:Cone_monoid}, which we omit.

\subsection{Recovering the fan}\label{SS:Recove}

Before establishing our results on compact and projective irrational toric varieties, we study certain limits in an
irrational toric variety $Y_\Sigma$ and show that the fan $\Sigma$ may be recovered from these limits.

Let $\Sigma\subset N$ be a fan.
Let \defcolor{$\varepsilon$} be the distinguished point in the dense orbit of $T_N$ on $Y_\Sigma$.
In every affine patch $V_\sigma$ for $\sigma\in\Sigma$, this restricts to the constant homomorphism.
If $L$ is the minimal cone of $\Sigma$ (its lineality space), then $\varepsilon=x_L$; our notation avoids $L$.
We study limits of $\varepsilon$ in $Y_\Sigma$ under one-parameter subgroups $\gamma_{sv}$ of $T_N$, giving a global
version of Lemma~\ref{L:limitOfDistinguishedPoints}. 

\begin{lemma}\label{L:limits_exist}
 Let $v\in N$.
 Then the family of translates $\gamma_{sv}.\varepsilon$ for $s\in \RR$ has a limit in $Y_\Sigma$ as $s\to\infty$ if and
 only if there is a cone $\sigma\in\Sigma$ with $v\in\sigma$.
 When this limit exists, it equals $x_\tau$, where $\tau$ is the minimal cone of $\Sigma$ that contains $v$, so that
 $v\in\tau^\circ$. 
\end{lemma}

\begin{proof}
 For $u\in M$ and $s\in\RR$, we have $(\gamma_{sv}.\varepsilon)(u)=\exp(-su\cdot v)$, as $\varepsilon(u)=1$.
 Thus
 \begin{equation}\label{Eq:limits}
   \lim_{s\to\infty}(\gamma_{sv}.\varepsilon)(u)\ =\ 
    \left\{\begin{array}{rcl}  0 &\ & \mbox{if }u\cdot v>0\\
                              1 &\ & \mbox{if }u\cdot v=0\\
                              \infty &\ & \mbox{if }u\cdot v<0\end{array}\right. .
 \end{equation}
 If there is a cone $\sigma\in\Sigma$ with $v\in\sigma$, then by~\eqref{Eq:limits} 
 $(\gamma_{sv}.\varepsilon)(u)$ has a limit as $s\to\infty$ for all $u\in\sigma^\vee$, and so the
 family $\gamma_{sv}.\varepsilon$ has a limit as $s\to\infty$ in the affine toric variety $V_\sigma$.
 (This is Lemma~\ref{L:limitOfDistinguishedPoints}.)
 Conversely, if $\gamma_{sv}.\varepsilon$ has a limit in $Y_\Sigma$ as $s\to\infty$, then there is an affine irrational toric variety
 $V_\sigma$ for some $\sigma\in\Sigma$ in which it has a limit.
 Then for all $u\in\sigma^\vee$, the family of real numbers 
 $(\gamma_{sv}.\varepsilon)(u)$ has a limit, which implies that $v\in\sigma$.

 The assertion identifying the limit is a consequence of the definition of $x_\tau$ 
 and of~\eqref{Eq:limits}. 
\end{proof}

Let $\Sigma\subset N$ be a fan.
Let $\defcolor{|\Sigma|}\subset N$ be the set of $v\in N$ such that $\gamma_{sv}.\varepsilon$ has a limit in $Y_\Sigma$ as
$s\to\infty$.
We define an equivalence relation on $|\Sigma|$.
For $v,w\in |\Sigma|$ we declare
\[
   v\ \sim\ w\ \mbox{ if and only if }\ 
   \lim_{s\to\infty}\gamma_{sv}.\varepsilon\ =\ 
   \lim_{s\to\infty}\gamma_{sw}.\varepsilon\,.
\]
By Lemma~\ref{L:limits_exist}, the set $|\Sigma|$ is the support of the fan $\Sigma$ and the equivalence classes are the
relative interiors of cones of $\Sigma$.
In fact, a cone $\sigma\in\Sigma$ is the closure of the set of $u\in |\Sigma|$ such that 
$\lim_{s\to\infty}\gamma_{sv}.\varepsilon = x_\sigma$.
Since these limits commute with the action of the torus $T_N$, we may replace $\varepsilon$ in the definition of $\sim$  by
any point $y$ in the dense orbit of $Y_\Sigma$.
Similarly, we may identify the cone $\sigma$ as the closure of 
the set of $u\in |\Sigma|$ such that for any $y$ in the dense orbit of $Y_\Sigma$, 
$\lim_{s\to\infty}\gamma_{sv}.y \in U_\sigma$.
We summarize this discussion.

\begin{corollary}\label{C:recoverFan}
 The fan $\Sigma\subset N$ may be recovered from the irrational toric variety $Y_\Sigma$ using limits under translation
 by one-parameter subgroups $\gamma_{sv}$ of elements $y$ in the dense orbit.
\end{corollary}

\subsection{Compact irrational toric varieties}\label{SS:Complete}

A classical toric variety $X_\Sigma$ is a proper scheme over $\spec\ZZ$ if and only if the rational fan $\Sigma$ is
\demph{complete} (every point of $N$ lies in some cone of $\Sigma$, so that $N=|\Sigma|$).
The analogous result holds for an irrational toric variety $Y_\Sigma$.

\begin{theorem}\label{Th:complete}
 Let $\Sigma\subset N$ be a fan.
 The irrational toric variety $Y_\Sigma$ is a compact topological space if and only if the fan $\Sigma$ is complete. 
\end{theorem}

\begin{proof}
 Suppose that $Y_\Sigma$ is compact.
 Recall the definitions and notation preceding Lemma~\ref{L:limits_exist}.
 As $Y_\Sigma$ is compact, for every $v\in N$, the family $\gamma_{sv}.\varepsilon$ has a limit in  $Y_\Sigma$ as
 $s\to\infty$.  
 By Lemma~\ref{L:limits_exist}, there is some cone $\sigma$ of $\Sigma$ with $v\in\sigma$,
 which implies that $\Sigma$ is a complete fan.

 Suppose now that the fan $\Sigma$ is complete.
 We prove that $Y_\Sigma$ is compact by showing that every sequence $\{y_n\mid n\in\NN\}\subset Y_\Sigma$ has a subsequence
 that converges in $Y_\Sigma$.
 To that end, we will replace $\{y_n\mid n\in\NN\}$ by a subsequence with desirable properties several times.
 By Theorem~\ref{Th:ITV_structure}, $Y_\Sigma$ is a disjoint union of finitely many orbits of $\RR_>^n$, one orbit
 $U_\sigma$ for each cone $\sigma$ of $\Sigma$.
 Thus there is some orbit $U_\sigma$ whose intersection with $\{y_n\mid n\in\NN\}$ is infinite.
 Replacing $\{y_n\mid n\in\NN\}$ by a subsequence, we may assume that $\{y_n\mid n\in\NN\}\subset U_\sigma$.
 By Corollary~\ref{C:Star}, $\overline{U_\sigma}=Y_{\Star(\sigma)}$.
 By its construction, if $\Sigma$ is complete, then $\Star(\sigma)$ is also complete.
 Replacing $Y_\Sigma$ by $Y_{\Star(\sigma)}$, we may assume that $\{y_n\mid n\in\NN\}$ lies in the dense orbit of $Y_\Sigma$.

 The dense orbit of $Y_\Sigma$ is parameterized by $N$ under the map $v\mapsto \gamma_v.\varepsilon$.
 For each $n\in\NN$, choose a point $v_n\in N$ such that $y_n=\gamma_{v_n}.\varepsilon$.
 This gives a sequence $\{v_n\mid n\in\NN\}\subset N$.
 Since $\Sigma$ is complete, $N$ is the finite disjoint union of the relative interiors $\sigma^\circ$ of cones $\sigma$ of
 $\Sigma$.
 There is some cone $\sigma$ such that $\sigma^\circ\cap\{v_n\mid n\in\NN\}$ is infinite.
 Replacing $\{v_n\mid n\in\NN\}$ by its intersection with $\sigma^\circ$, we may assume that 
 $\{v_n\mid n\in\NN\}\subset\sigma^\circ$. 

 For a face $\tau$ of $\sigma$, we say that $\{v_n\mid n\in\NN\}$ is \demph{$\tau$-bounded} if its image in the quotient 
 $N/\langle\tau\rangle$ has a bounded subsequence.
 By~\cite[Lem.~1.3]{PSV}, the set of faces $\tau$ of $\sigma$ for which $\{v_n\mid n\in\NN\}$ is $\tau$-bounded forms an
 order ideal. 
 Let $\tau$ be a minimal face of $\sigma$ for which $\{v_n\mid n\in\NN\}$ is $\tau$-bounded.
 Replace $\{v_n\mid n\in\NN\}$ by a subsequence whose image in $N/\langle\tau\rangle$ is bounded.
 As in~\cite[Ex.~4.1]{PSV}, there is a bounded set $B\subset \sigma\subset N$ such that 
 $\{v_n\mid n\in\NN\}\subset B+\tau$, and thus there are sequences
 $\{b_n\mid n\in\NN\}\subset B$ and $\{c_n\mid n\in\NN\}\subset\tau$ with $v_n=b_n+c_n$.
 Since $\{b_n\mid n\in\NN\}$ is bounded, we may further pass to a convergent subsequence with limit $b\in\sigma$.

 Replacing all sequences by their corresponding subsequences, we claim that in $V_\tau$, 
 \begin{equation}\label{Eq:main_limit}
  \lim_{n\to\infty} y_n\ =\ \gamma_b.x_\tau\ \in\ U_\tau\,,
 \end{equation}
 which will complete the proof.

 Consider the sequence $\{y_n\mid n\in\NN\}$ as a subset of $V_\tau=\Hom_c(\tau^\vee,\RR_\geq)$.
 For $u\in\tau^\vee$, the proof of Lemma~\ref{L:limits_exist} shows that $y_n(u) =\exp(-u\cdot v_n)$, as
 $y_n=\gamma_{v_n}.\varepsilon$. 
 Since $v_n=b_n+c_n$,  
\[
  y_n(u)\ =\ 
  \exp(- u\cdot v_n)\ =\  \exp(- u\cdot b_n)\cdot \exp(- u\cdot c_n)\,.
\]
 If $u\in\tau^\perp$, then $u\cdot c_n=0$, so that 
\[
  \lim_{n\to\infty} y_n(u)\ =\
  \lim_{n\to\infty}\exp(- u\cdot b_n)\ =\  \exp(- u\cdot b)\,.
\]
 If $u\in\tau^\vee{\smallsetminus}\tau^\perp$, then $u$ exposes a proper face of $\tau$, and 
 the minimality of $\tau$ implies that $u\cdot v_n$ has no bounded subsequence.
 But then $u\cdot c_n$ has no bounded subsequence.
 Since $u\cdot c_n\geq 0$, we conclude that $u\cdot c_n$ has limit $+\infty$ as $n\to\infty$.
 Thus 
\[
  \lim_{n\to\infty} y_n(u)\ =\
  \lim_{n\to\infty}  \exp(- u\cdot v_n)\ =\  0\,.
\]
 These calculations together establish~\eqref{Eq:main_limit}, and complete the proof.
\end{proof}

\subsection{Projective irrational toric varieties}\label{SS:Projective}


The analog of projective space for irrational toric varieties is the standard simplex
 \begin{equation}\label{Eq:simplex}
   \defcolor{\bsimplex^n}\ :=\ 
   \bigl\{ (u_0,u_1,\dotsc,u_n)\in\RR^{n+1}_\geq \mid {\textstyle \sum} u_i = 1\bigr\}\,.
 \end{equation}
As any ray in the orthant $\RR^{n+1}_\geq$ meets $\simplex^n$ in a unique point, we may identify $\simplex^n$ with the set
of rays.
This implies that $\simplex^n=(\RR^{n+1}_\geq{\smallsetminus}\{0\})/\RR_>$, the quotient under multiplication by positive
scalars.
As in \S~\ref{S:Affine}, it is convenient to write $\simplex^\calA\subset\RR^\calA_\geq$, where $\calA\subset M$ is 
finite.

\begin{example} \label{Ex:ProjectiveSpace}
 The simplex has the structure of an irrational toric variety associated to a fan.
 For this, let $\defcolor{[n]}:=\{0,1,\dotsc,n\}$ and $e_0,\dotsc,e_n$ be the standard basis for
 $\RR^{[n]}\simeq\RR^{n+1}$.
 Letting  $f_0,\dotsc,f_n\in\RR^{[n]}$ be the dual basis, the standard simplex
 $\defcolor{\bsimplex^{[n]}}\subset\RR^{[n]}$ is their convex hull, $\conv\{f_0,\dotsc,f_n\}$, which equals the 
 standard simplex $\simplex^n$~\eqref{Eq:simplex}.
 Define a complete fan $\Sigma_{[n]}\subset\RR^{[n]}$ with one cone $\sigma_I$ for each proper subset 
 $I\subsetneq[n]$ defined by 
\[
   \defcolor{\sigma_I}\ :=\ \cone\{e_i\mid i\in I\}\ +\ \RR \bbI\,,
\]
 where $\defcolor{\bbI}:=e_0+\dotsb+e_n$.
 Note that $\sigma_\emptyset=\RR\bbI$.
 We have the irrational toric variety $Y_{\Sigma_{[n]}}$.
 
 The hyperplane $\{u\in \RR^{[n]}\mid u\cdot \bbI =0\}$ contains all dual cones $\sigma^\vee_I$ and is spanned by the
 differences $f_i{-}f_j$ for $i,j\in [n]$.
 Set $\defcolor{V_I}:=\Hom_c(\sigma_I^\vee,\RR_\geq)$.
 For $I\subsetneq[n]$,  $j\not\in I$, and $\varphi\in V_I$, set 
\[
   \psi_I(\varphi)\ =\ \RR_\geq\Bigl( \sum_{i\in [n]} \varphi(f_i{-}f_j) f_i \Bigr)\ \cap\ \simplex^{[n]}\,,
\]
 the intersection of the ray through $\sum_{i} \varphi(f_i{-}f_j)f_i$ with the simplex $\simplex^{[n]}$.
 This injective map does not depend on the choice of $j\not\in I$ and defines a map $\psi_I\colon V_I\to\simplex^{[n]}$.
 These maps $\psi_I$ are compatible with the gluing in that if $y\in Y_{\Sigma_{[n]}}$ lies in two affine patches $V_I$ and
 $V_J$, then $\psi_I(y)=\psi_J(y)$. 
 Thus these maps induce a homeomorphism $\Psi\colon Y_{\Sigma_{[n]}}\xrightarrow{\sim}\simplex^{[n]}$.
 (This will be explained in greater generality in the proof of Theorem~\ref{Th:ProjITV}.)

 The quotient map  $\RR^{[n]}_\geq{\smallsetminus}\{0\}\twoheadrightarrow\simplex^{[n]}$ may be understood in terms of
 Theorem~\ref{Th:mapsOfFans}. 
 Let $\defcolor{\Sigma'_{[n]}}\subset\RR^{[n]}$ be the fan consisting of the boundary of the nonnegative orthant.
 Its cones are 
\[
   \sigma'_I\ :=\ \cone\{ e_i\mid i\in I\}\,,
\]
 for all proper subsets $I\subsetneq[n]$.
 The irrational toric variety $Y_{\Sigma'_{[n]}}$ is $\RR^{n+1}_\geq{\smallsetminus}\{0\}$, as the orbit consisting of the origin in
 $\RR^{n+1}_\geq$ corresponds to the omitted full-dimensional cone.
 As $\sigma'_I\subset \sigma_I$, these inclusions induce a map of fans $\Psi\colon\Sigma'_{[n]}\to\Sigma_{[n]}$, and thus
 a functorial map of toric varieties $Y_{\Sigma'_{[n]}}\to Y_{\Sigma_{[n]}}$, which is the 
 quotient map  $\RR^{[n]}_\geq{\smallsetminus}\{0\}\twoheadrightarrow\simplex^{[n]}$.\hfill$\diamond$
\end{example}

Suppose that $\calA\subset M$ lies on an affine hyperplane in that there is some $v\in N$ and $0\neq r\in\RR$ with 
$a\cdot v=r$ for all $a\in\calA$.
Thus for $t\in T_N$ and $s\in\RR$, we have 
\[
  \varphi_\calA(\gamma_{sv}.t)\ =\ (\gamma_{sv}(a) t^a\mid a\in\calA)
    \ =\ \exp(-sr) \varphi_\calA(t)\,,
\]
as $\gamma_{sv}(a)=\exp(-s a\cdot v)=\exp(-sr)$ for $a\in\calA$ and $s\in\RR$.
Consequently, the affine irrational toric variety $Y_\calA\subset\RR^\calA_\geq$ is a union of rays.
Define the \demph{projective irrational toric variety} \defcolor{$Z_\calA$} to be the intersection 
$Y_\calA\cap\simplex^{\calA}$, equivalently, the quotient $(Y_\calA\smallsetminus\{0\})/\RR_>$ under scalar multiplication.
This has an action of $T_N$ with a dense orbit, as the action of $T_N$ on $\RR^\calA_\geq$ through $\varphi_\calA$ gives an
action on rays and hence on $\simplex^\calA$, which restricts to an action on $Z_\calA$.

The restriction of the tautological map $\pi_\calA\colon\RR^\calA_\geq\to\cone(\calA)$ to the simplex $\simplex^\calA$ is
the canonical parametrization of the convex hull of $\calA$,
\[
   \simplex^\calA\ni u\ \longmapsto\ \sum_a u_a a\ \in\ \conv(\calA)\,.
\]
By Birch's Theorem (Proposition~\ref{P:Birch}), restricting to the projective toric variety $Z_\calA$ gives a homeomorphism
$\pi_\calA\colon Z_\calA\xrightarrow{\sim}\conv(\calA)$, called the \demph{algebraic moment map}~\cite{vilnius}.
This isomorphism is also essentially proven by Krasauskas~\cite[Thm.~26]{Krasauskas}.

A \demph{polytope} $P\subset M$ is a set that is the convex hull, $P=\conv(\calA)$, of a finite subset $\calA\subset M$.
For any $v\in N$, the subset $F$ of $P$ where the linear function $u\mapsto u\cdot v$ is minimized is the 
\demph{face exposed} by $v$, and every face of $P$ is exposed by some element of $N$.
For example, the polytope $P$ is exposed by the zero vector.

For a face $F$ of $P$, the set of $v\in N$ that expose a face containing $F$ forms a polyhedral
cone \defcolor{$\sigma_F$} in $N$, whose relative interior consists of those $v$ which expose $F$.
The faces of $\sigma_F$ are cones $\sigma_E$ for $E$ a face of $P$ containing $F$.
The collection of these cones $\sigma_F$ for the faces $F$ of $P$ forms the (inner) \demph{normal fan} to $\Delta$, which
is a complete fan.
For example, 
the fan $\Sigma_{[n]}$ of
Example~\ref{Ex:ProjectiveSpace} is the normal fan of the standard simplex $\simplex^{[n]}$.
Figure~\ref{F:NormalFans} shows two views of a polytope and its normal fan.
\begin{figure}[htb]
  \centering
  \raisebox{10pt}{\includegraphics[height=70pt]{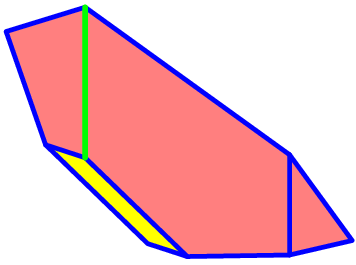}}\qquad
  \includegraphics[height=90pt]{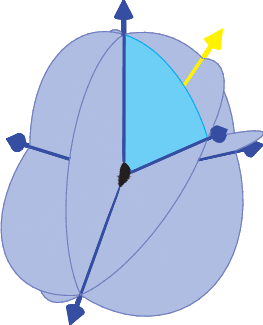}\qquad\qquad
  \raisebox{5pt}{\includegraphics[height=80pt]{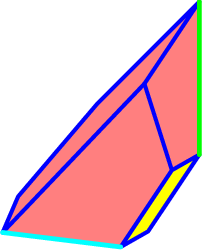}}\qquad
  \includegraphics[height=90pt]{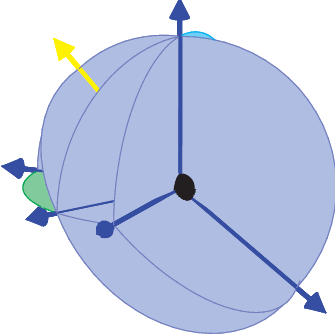}
  \caption{Two views of a polytope and its normal fan.}
  \label{F:NormalFans}
\end{figure}
The yellow ray in the normal fan exposes the yellow facet, 
the green cone exposes the green edge, and
the cyan cone exposes the cyan edge.

\begin{theorem}\label{Th:ProjITV}
 Suppose that $P\subset M$ is a polytope lying on an affine hyperplane with normal fan $\Sigma$.
 For any $\calA\subset P$ with $\conv(\calA)=P$, there is an injective map of irrational toric varieties 
 $\Psi_\calA\colon Y_\Sigma\to\simplex^\calA$ whose image is the projective irrational toric variety $Z_\calA$.
 The map $\Psi_\calA$ composed with the algebraic moment map $\pi_\calA$ is a homeomorphism
 $Y_\Sigma\xrightarrow{\sim}P$.
\end{theorem}

When $\Sigma$ is a rational fan, this is a standard result about the nonnegative part of projective toric varieties,
polytopes with integer vertices, and the classical moment map.

\begin{proof}
 Let $\calA\subset P$ be a subset with $\conv(\calA)=P$.
 For a face $F$ of $P$, we have $\conv(F\cap\calA)=F$.
 An intersection $\calF:=F\cap\calA$ for a face $F$ of $P$ a \demph{face} of $\calA$.
 Cones of the normal fan $\Sigma$ to $P$ correspond to faces $\calF$ of $\calA$.
 For a face $\calF$ of $\calA$, the corresponding cone in $\Sigma$ is 
\[
   \defcolor{\sigma_\calF}\ =\ \{ v\in N\mid f\cdot v \leq a\cdot v\ \mbox{for all } f\in\calF\mbox{ and } a\in\calA\}\,.
\]
 The lineality space of $\Sigma$ is spanned by those $v\in N$ such that $a\cdot v=b\cdot v$ for any $a,b\in\calA$.

 For a subset $\calB\subset\calA$ and any $u\in M$, we define $\defcolor{\calB{-}u}:=\{b{-}u\mid b\in\calB\}$. 
 Duals of cones $\sigma_\calF$ lie in the subspace $L$ of $M$ spanned by the differences $\{b{-}a\mid a,b\in\calA\}$,
 equivalently by $\calA{-}a$ for any $a\in\calA$.
 For a face $\calF$ of $\calA$, the cone dual to $\sigma_\calF$ is 
\[
   \sigma_\calF^\vee\ =\ 
     \cone(\calA{-}f)\ +\ \RR(\calF{-}f)\,,
\]
 for any $f\in\calF$.
 Choosing another $f'\in\calF$ translates the points $\calA{-}f$ along the lineality space $\RR(\calF{-}f)$.
 Figure~\ref{F:calFdual} shows an example of $\sigma_\calF^\vee$.
\begin{figure}[htb]
 \centering
  \begin{picture}(170,90)
    \put(0,0){\includegraphics{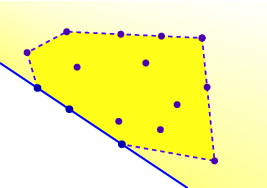}}
    \put(110,58){$\sigma_\calF^\vee$}
    \put(51,45){$\calA{-}f$}

    \put(11.2,21){\vector(1,4){5.5}}   
     \put(20,20){\vector(3,4){11}}    
    \put(29,13.5){\vector(4,1){25}}
    \put(55,9){$0$}
    \put(0,10){$\calF{-}f$}    
    \put(139,2){$\RR(\calF{-}f)$}   \put(137,6){\vector(-1,0){51}}
  \end{picture}

 \caption{A dual cone $\sigma_\calF^\vee$.}
 \label{F:calFdual}

\end{figure}
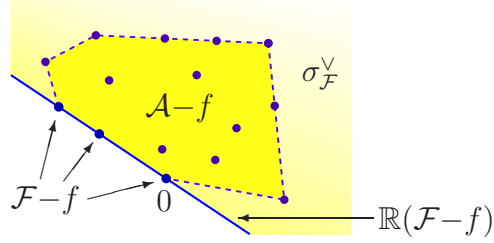
 The affine toric variety $V_\calF$ corresponding to a face $\calF$ of $\calA$ is $\Hom_c(\sigma_\calF^\vee,\RR_\geq)$. 

 Let $f\in\calF$ and consider the map
\[
   \defcolor{\psi_f}\ \colon\ V_\calF\ \ni\ \varphi\ \longmapsto\ 
     ( \varphi(a{-}f)\mid a\in\calA)\ \in\ \RR^\calA_\geq\,.
\]
 Since for $f,f'\in\calF$, $a\in\calA$, and $\varphi\in V_\calF$, we have $\varphi(a{-}f)=\varphi(f'{-}f)\varphi(a{-}f')$,
 it follows that $\psi_{f}(\varphi)=\varphi(f'{-}f)\psi_{f'}(\varphi)$.
 That is, the two points $\psi_{f}(\varphi)$ and $\psi_{f'}(\varphi)$ lie along the same ray in $\RR^\calA_\geq$.
 Consequently, the map  $V_\calF\to\simplex^\calA$ defined by 
 \begin{equation}\label{Eq:psicalF}
   V_\calF\ \ni\ \varphi\ \longmapsto\ (\RR_\geq\psi_f(\varphi))\cap\simplex^\calA\ \in\ \simplex^\calA\,,
 \end{equation}
 is independent of the choice of $f\in\calF$.
 Write \defcolor{$\psi_\calF$} for the map~\eqref{Eq:psicalF}, which is a continuous injection from $V_\calF$ into 
 $\simplex^\calA$.

 Suppose that $\calF,\calG$ are faces of $\calA$ with $\calF$ a face of $\calG$.
 Then $V_\calG\subset V_\calF$, and for $\varphi\in V_\calG$, we have $\psi_\calF(\varphi)=\psi_\calG(\varphi)$, as both
 maps are computed using $\psi_f$ for $f\in\calF\subset\calG$.
 Thus the maps $\psi_\calF$ for $\calF$ a face of $\calA$ are compatible with the gluing of the 
 $V_\calF$ to form $Y_\Sigma$, and so they induce a continuous map $\defcolor{\Psi_\calA}\colon Y_\Sigma\to\simplex^\calA$.
 This is an injection because if $\calF,\calG$ are faces of $\calA$ that are not faces of each other, then the support of 
 $\varphi\in V_\calF\smallsetminus V_\calG$ contains $\calF$ and is disjoint from $\calG{\smallsetminus}\calF$.
 Consequently, $\Psi_\calA$ is injective on the union $V_\calF\cup V_\calG$.

 We claim that  $\Psi_\calA(Y_\Sigma)=Z_\calA$.
 Since both are complete, it suffices to show that both contain the same dense subset.
 Let $t=\gamma_v\in T_N$ with $v\in N$.
 Since, for $u\in M$, $t^u=\exp(-u\cdot v)$, we have $\varphi_\calA(t)=(\exp(-a\cdot v)\mid a\in\calA)$.
 This lies on a ray in $\RR^\calA_\geq$ that meets $\simplex^\calA$ in the point
 \begin{equation}\label{Eq:varphiZ}
   \bigl(\RR_\geq\cdot(\exp(-a\cdot v)\mid a\in\calA)\bigr)\ \cap\ \simplex^\calA\ \in\ Z_\calA\,.
 \end{equation}

 The corresponding point $\gamma_v.\varepsilon$ of $Y_\Sigma$ lies in
 $V_\calA=\Hom_c(L,\RR_>)$, where $L=\RR(\calA{-}b)$ for any $b\in\calA$.
 Its image $\Psi_\calA(\gamma_v.\varepsilon)$ is
\[
   \bigl(\RR_\geq\cdot(\gamma_v.\varepsilon(a{-}b)\mid a\in\calA)\bigr)\ \cap\ \simplex^\calA\ =\ 
   \bigl(\RR_\geq\cdot(\exp(-a\cdot v)\mid a\in\calA)\bigr)\ \cap\ \simplex^\calA\,,
\]
 as $\varepsilon(a-b)=1$ and $\gamma_v(a-b)=\exp(-a\cdot v)\exp(b\cdot v)$, so that the two rays are equal.
 Comparing this to~\eqref{Eq:varphiZ} completes the proof.
\end{proof}

The restriction that $\calA$ lie in an affine hyperplane of $M$ may be removed by enlarging $N$ and $M$ by adding
a summand of $\RR$ to each and placing $\calA\subset M\oplus\RR$ in the copy of $M$ at height $1$.
That is, as the set of points  $\{(a,1)\mid a\in\calA\}\subset M\oplus\{1\}\subset M\oplus\RR$.

\section{Hausdorff Limits and the Secondary Polytope}
\label{S:Hausdorff}

We use the theory developed in the previous sections to establish our main result about the moduli space of limits of
translations of an irrational projective toric variety.
The set \defcolor{$\closed(X)$} of closed subsets of a compact metric space $X$ is itself a compact metric space.
The Hausdorff distance between $Y,Z\in\closed(X)$ is
\[
   \defcolor{d_H(Y,Z)}\ :=\    
    \max\{\max_{y\in Y} \min_{z\in Z} d(y,z)\, ,\, 
    \max_{z\in Z} \min_{y\in Y} d(y,z)\}\,,
\]
where $d(\bullet,\bullet)$ is the metric on $X$~\cite[p.~279]{Munkres}.

Suppose that $\calA\subset M$ lies on an affine hyperplane so that the irrational affine toric variety
$Y_\calA\subset\RR^\calA_\geq$ is a union of rays with associated irrational projective toric variety
$Z_\calA=Y_\calA\cap\simplex^\calA$.
The positive torus $\RR^\calA_>$ acts linearly on the orthant $\RR^\calA_\geq$ by scaling each coordinate, 
\[  
    w.x\ =\ w.(x_a\mid a\in\calA)\ =\ (w_ax_a\mid a\in\calA)\,,
\]
for $w\in\RR^\calA_>$ and $x\in\RR^\calA_\geq$.
This induces an action of $\RR^\calA_>$ on rays, and hence on  $\simplex^\calA$. 

For $w\in\RR^\calA_>$, the translate $w.Z_\calA$ is a closed subset of $\simplex^\calA$ defined by binomials similar to
those defining $Y_\calA$ of Proposition~\ref{P:IATV} (the actual binomials are described in~\cite[Prop.~A2]{GSZ}).
The association $w\mapsto w.Z_\calA$ gives a continuous  map $\RR^\calA_>\to\closed(\simplex^\calA)$.
Let \defcolor{$\Delta_\calA$} be the closure of that image.
This is a compact Hausdorff space equipped  with a continuous action of $\RR^\calA_>$, and it consists of all Hausdorff limits of translates
of $Z_\calA$. 
The main result of~\cite{PSV} was a set-theoretic identification of the points of $\Delta_\calA$.
We extend that to construct a homeomorphism between $\Delta_\calA$ and the irrational
toric variety associated to the secondary fan of $\calA$.
This fan is normal to the secondary polytope of $\calA$, so Theorem~\ref{Th:ProjITV} identifies $\Delta_\calA$ with the
secondary polytope.
This extends the observation of~\cite{GSZ} which adapted the results of~\cite{KSZ1,KSZ2} to prove this statement
when $\calA\subset M_\ZZ$ is integral.
 
We define the secondary fan and secondary polytope of $\calA$, which were introduced in~\cite{GKZ} 
(see also~\cite{dLRS}).
For $\lambda\in\RR^\calA$, let $Q_\lambda\subset M\oplus\RR$ be the convex hull
of the lifted points
\[
   \defcolor{Q_\lambda}\ :=\    
    \conv\{ (a,\lambda_a)\in M\oplus\RR \mid a\in\calA\}\ .
\]
A \demph{lower face} of $Q_\lambda$ is a face with an (inner) normal vector whose last coordinate is positive.
We define a system \defcolor{$S(\lambda)$} of subsets of $\calA$.
A subset $\calF\subset\calA$ is an element of $S(\lambda)$ if there is a lower face $F$ of $Q_\lambda$ such that 
\[
   \calF\ =\ \{a\in\calA \mid (a,\lambda_a)\in F\}\,.
\]
A system $S$ of subsets is a \demph{regular subdivision of $\calA$} if $S=S(\lambda)$ for some
$\lambda\in\RR^\calA$. 
Elements $\calF$ of a regular subdivision $S$ are its \demph{faces}.
Figure~\ref{F:Regular} shows four regular subdivisions of a set $\{a,b,c,d,e\}$ of points in the plane, including
\begin{figure}[htb]
 \centering
   \begin{picture}(118,90)
     \put(0,0){\includegraphics[height=90pt]{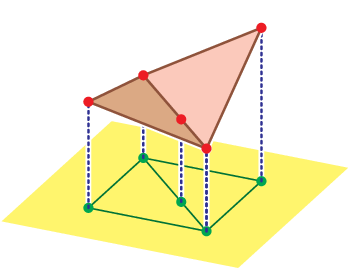}}
     \put(21,17){\small$a$} \put(40,36){\small$b$} \put(51,21){\small$c$} 
     \put(69,5){\small$d$} \put(91,26){\small$e$}
   \end{picture}%
   \begin{picture}(118,90)
     \put(0,0){\includegraphics[height=90pt]{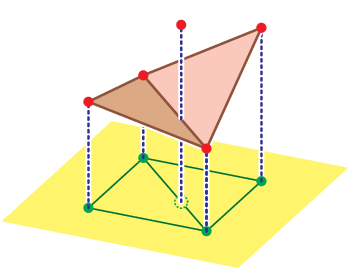}}
     \put(21,17){\small$a$} \put(40,36){\small$b$} \put(51,21){\small$c$} 
     \put(69,5){\small$d$} \put(91,26){\small$e$}
   \end{picture}%
   \begin{picture}(118,90)
     \put(0,0){\includegraphics[height=90pt]{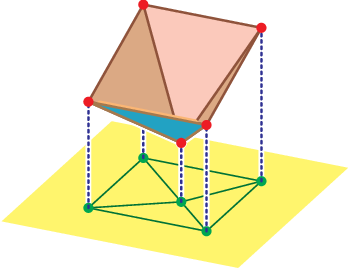}}
     \put(21,17){\small$a$} \put(40,36){\small$b$} \put(50,23.5){\small$c$} 
     \put(69,5){\small$d$} \put(91,26){\small$e$}
   \end{picture}%
   \begin{picture}(118,90)
     \put(0,0){\includegraphics[height=90pt]{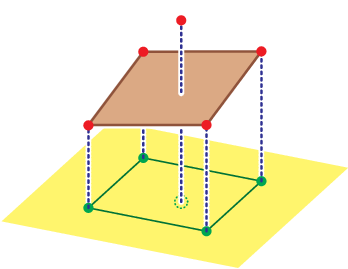}}
     \put(21,17){\small$a$} \put(40,36){\small$b$} \put(52,21){\small$c$} 
     \put(69,5){\small$d$} \put(91,26){\small$e$}
   \end{picture}
 \caption{Regular subdivisions of a point set.}
 \label{F:Regular}
\end{figure}
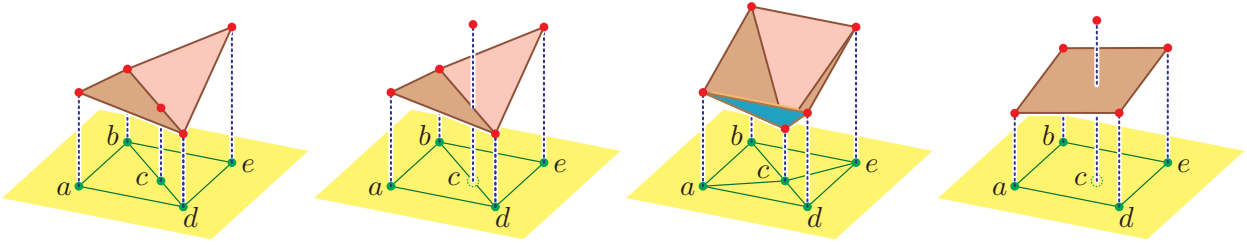
the lower faces of the lifted polytope $Q_\lambda$.
We list the facets (maximal elements) of each subdivision.
The first on the left has two facets, $\{a,b,c,d\}$ and $\{b,c,d,e\}$, as the lift of the point $c$ is collinear with the
lifts of $b$ and $d$.
The point $c$ is lifted above the lower hull in the second and does not participate in this regular subdivision, whose
facets are  $\{a,b,d\}$ and $\{b,d,e\}$.  
The third has four facets, $\{a,b,c\}$, $\{a,c,d\}$, $\{b,c,e\}$, and $\{c,d,e\}$.
The fourth has one facet $\{a,b,d,e\}$, as their lifts are coplanar and the point $c$ is again lifted above the
lower hull.

Two elements $\lambda,\mu\in\RR^\calA$ are equivalent if they induce the same regular
subdivision of $\calA$, $S(\lambda)=S(\mu)$.
An equivalence class is defined by finitely many linear equations and linear inequalities, and so the closure of each
equivalence class is a polyhedral cone.
These cones fit together to form the \demph{secondary fan $\Sigma(\calA)$} of the point configuration $\calA$.

A face $\calF\in S(\lambda)$ corresponds to a lower face $F$ of the lifted polytope $Q_\lambda$,
with  $F=\conv\{(f,\lambda_f)\mid f\in\calF\}$.
The projections of the lower faces of $Q_\lambda$ to $M$ form a regular polyhedral subdivision of $\conv(\calA)$ whose
polytopes are $\conv(\calF)$ for $\calF\in S(\lambda)$.  
A regular subdivision $S(\lambda)$ is a \demph{triangulation} if for each face $\calF\in S(\lambda)$, its convex hull is a 
simplex with vertices $\calF$.
Triangulations correspond to full-dimensional cones of $\Sigma(\calA)$.
The middle two subdivisions in Figure~\ref{F:Regular} are triangulations.

The secondary fan $\Sigma(\calA)$ is the normal fan of the \demph{secondary polytope $P(\calA)$} of
$\calA$. 
It lies in $\RR^\calA$ and its vertices correspond to regular triangulations $\calT$ of $\calA$.
The $a$th coordinate of the vertex corresponding to $\calT$ is the sum of the volumes of the convex
hulls of faces $\calF\in\calT$ that contain $a$.
When $a$ does not participate in $\calT$, the vertex has $a$th coordinate 0.

A subset $\calF\subset\calA$ corresponds to a face $\simplex^\calF\subset\simplex^\calA$; it
is the set of points $x\in\simplex^\calA$ with $x_a=0$ for $a\not\in\calF$.
As in Section~\ref{S:Affine}, the irrational toric variety $Z_\calF$ lies in this face $\simplex^\calF$.
For a regular subdivision $S$ of $\calA$, we define a \demph{complex $Z(S)$} of irrational toric varieties to be
 \begin{equation}\label{Eq:complex}
    Z(S)\ :=\ \bigcup_{\calF\in S}  Z_\calF\ .
 \end{equation}
This is a complex in that if $\calF,\calF'\in S$ with $\emptyset\neq\calG=\calF\cap\calF'$, so that $\calG$ is also a face
of $S$, then 
\[
    Z_\calG\ =\ Z_\calF \cap Z_{\calF'}\, .
\]

An element $w\in\RR^\calA_>$ acts on the projective toric variety $Z_\calF$ as before, with
$w.Z_\calF\subset\simplex^\calF$.
Only the coordinates $w_f$ of $w$ for $f\in\calF$ act on $\simplex^\calF$.
For a regular subdivision $S$ of $\calA$ and $w\in\RR^\calA_>$, we have the translated complex of irrational toric
varieties, 
\[
   \defcolor{Z(S,w)}\ :=\ \bigcup_{\calF\in S}  w.Z_\calF\ .
\]
We recall one of the main results of~\cite{PSV}.

\begin{proposition}[{\cite[Thm.~3.3]{PSV}}]
 The points of $\Delta_\calA\subset\closed(\simplex^\calA)$ are exactly the translated complexes of irrational toric
 varieties $Z(S,w)$ for $w\in\RR^\calA_>$ and $S$ a regular subdivision of $\calA$.
\end{proposition}

Recall that $\varepsilon\in Y_{\Sigma(\calA)}$ is the distinguished point of its dense $\RR^\calA_>$-orbit.
We give our main theorem.

\begin{theorem}\label{Th:SecondaryPolytope}
 For $w\in\RR^\calA_>$, the association $\psi\colon w.Z_\calA\mapsto w.\varepsilon$ is a well-defined continuous map from the
 set of translates of $Z_\calA$ to the dense orbit of $Y_{\Sigma(\calA)}$ that extends to
 an $\RR^\calA_>$-equivariant homeomorphism 
 $\Delta_\calA\xrightarrow{\sim}Y_{\Sigma(\calA)}$. 
 Composing it with an algebraic moment map $Y_{\Sigma(\calA)}\to P(\calA)$ gives a homeomorphism between $\Delta_\calA$ and
 the secondary polytope $P(\calA)$.
\end{theorem}

\begin{proof}
 We first extend $\psi$ to an $\RR^\calA_>$-equivariant bijection between $\Delta_\calA$ and
 $Y_{\Sigma(\calA)}$, which shows that it is well-defined and that it is a homeomorphism on orbits of
 $\RR^\calA_>$. 
 To show that $\psi$ is a homeomorphism between $\Delta_\calA$ and $Y_{\Sigma(\calA)}$, we will use the proof of
 Theorem~\ref{Th:complete} and Section 4.1 of~\cite{PSV}.
 The last statement is Theorem~\ref{Th:ProjITV}, as $\Sigma(\calA)$ is the normal fan to $P(\calA)$.

 For the toric variety $Y_{\Sigma(\calA)}$, we have $N=\RR^\calA$ as this is the ambient space for the fan $\Sigma(\calA)$.
 Its dual space $M$ is naturally identified also with $\RR^\calA$ under the usual Euclidean dot product.
 For $v\in N$, the element $\gamma_v\in T_N$ is defined~\eqref{Eq:gammav} for $u\in M$ by
 $\gamma_v(u)=\exp(-u\cdot v)$.
 This identifies $N=\RR^\calA$ with $\RR^\calA_>$ where $(v_a\mid a\in\calA)\mapsto (e^{-v_a}\mid a\in\calA)$.

 Let $\sigma$ be a cone of $\Sigma(\calA)$ with corresponding regular subdivision \defcolor{$S_\sigma$}.
 The orbit $U_\sigma$ in  $Y_{\Sigma(\calA)}$ has distinguished point  $x_\sigma$.
 Under the map $\RR^\calA\xrightarrow{\sim}\RR^\calA_>$ given by $v\mapsto \gamma_v$, the orbit
 $U_\sigma=\RR^\calA_>.x_\sigma$ is identified with $\RR^\calA/\langle\sigma\rangle$, so that the stabilizer of $x_\sigma$
 is the linear span $\langle\sigma\rangle$ of $\sigma$.
 This identification is as a topological space, and as a $\RR^\calA_>$ or $\RR^\calA$-orbit.

 The complex $Z(S_\sigma)=Z(S_\sigma,1)$  for the regular subdivision $S_\sigma$ is a distinguished point of $\Delta_\calA$
 corresponding to $\sigma$. 
 By Lemma~2.4 of~\cite{PSV}, the stabilizer of $Z(S_\sigma,1)$ in $\RR^\calA$ is also $\langle\sigma\rangle$.
 Again, the orbit of $Z(S_\sigma,1)$ is identified with  $\RR^\calA/\langle\sigma\rangle$, as a topological space.
 Thus the association $\psi\colon Z(S_\sigma,\gamma_v)\mapsto \gamma_v.x_\sigma$ is a well-defined map.
 This induces an $\RR^\calA_>$-equivariant homeomorphism between the orbits in both spaces, because when both orbits are
 identified with $\RR^\calA/\langle\sigma\rangle$, $\psi$ becomes the identity map.
 When the subdivision $S_\sigma$ has a single facet $\calA$, so that $\sigma$ is the minimal cone of $\Sigma(\calA)$, this
 is the map $\psi$ in the statement of the theorem.

 Write $\psi\colon\Delta_\calA\to Y_{\Sigma(\calA)}$ for the bijection which is given by
\[
   \psi\ \colon\   Z(S_\sigma,w)\ \longmapsto\  w.x_\sigma\,,
\]
 for $w\in\RR^\calA_>$ and $\sigma$ a cone of $\Sigma(\calA)$.
 We show that if $\{w_n\mid n\in\NN\}\subset\RR^\calA_>$, $w\in\RR^\calA_>$, and $\sigma\in\Sigma(\calA)$ are such that 
 \begin{equation}\label{Eq:A}
    \lim_{n\to\infty} w_n.Z_\calA\ =\ Z(S_\sigma,w) 
 \end{equation}
 in the space $\Delta_\calA$ of Hausdorff limits, then 
 \begin{equation}\label{Eq:B}
    \lim_{n\to\infty} w_n.\varepsilon\ =\ w. x_\sigma 
 \end{equation}
 in the irrational toric variety $Y_{\Sigma(\calA)}$, and vice versa.
 As both $\Delta_\calA$ and $Y_{\Sigma(\calA)}$ are $\RR^\calA_>$-equivariant closures of their dense orbits, this will
 complete the proof.
 A more granular proof could work on pairs of corresponding orbits in the two spaces, arguing on each
 component $Z_\calF$ and using the recursive structure of both sets and of regular subdivisions and refinements.

 Our argument that $\psi$ preserves limits of sequences follows from the proofs of Theorem~3.3 in~\cite{PSV} and  
 Theorem~\ref{Th:complete} (which are essentially the same).
 These proofs show that the sequences of translates $w_n.Z_\calA$ and $w_n.\varepsilon$ have convergent subsequences.
 In each, we replace $w_n\in\RR^\calA_>$ by $v_n\in\RR^\calA$, where $w_n=\gamma_{v_n}$, and then
 replace $\{v_n\mid n\in\NN\}$ by any subsequence $\{v_n\mid n\in\NN\}\cap\sigma$, where $\sigma$ is a cone of $\Sigma(\calA)$ 
 that meets  $\{v_n\mid n\in\NN\}$ in an infinite set.
 Finally, $\tau$ is a minimal face of $\sigma$ such that $\{v_n\mid n\in\NN\}$ is bounded in
 $\RR^\calA/\langle\tau\rangle$, and then $v\in\RR^\calA$ satisfies
\[
    v\ +\ \langle\tau\rangle\ \mbox{ is an accumulation point of } 
    \{v_n + \langle\tau\rangle\mid n\in\NN\}\ \mbox{ in }\ \RR^\calA/\langle\tau\rangle\,.
\]
 Restricting to subsequences, we have $\lim_{n\to\infty}w_n.\varepsilon=v.x_\tau$ and
 $\lim_{n\to\infty}w_n.Z_\calA=Z(S_\tau,v)$. 

 We assumed that the original limits (\eqref{Eq:A} or~\eqref{Eq:B})  exist without restricting to subsequences.
 Thus, for every cone $\sigma$ with $\sigma\cap\{v_n\mid n\in\NN\}$ infinite,  $\tau$ is the unique minimal face of
 $\sigma$ such that $\{v_n\mid n\in\NN\}$ is bounded in 
 $\RR^\calA/\langle\tau\rangle$, and then $v$ is a point such that  
\[
    \lim_{n\to\infty} v_n + \langle\tau\rangle\ =\ v\ +\ \langle\tau\rangle\,.
\]
 This completes the proof, as both $\psi$ and $\psi^{-1}$ preserve limits of sequences.
\end{proof}

\providecommand{\bysame}{\leavevmode\hbox to3em{\hrulefill}\thinspace}
\providecommand{\MR}{\relax\ifhmode\unskip\space\fi MR }
\providecommand{\MRhref}[2]{%
  \href{http://www.ams.org/mathscinet-getitem?mr=#1}{#2}
}
\providecommand{\href}[2]{#2}

\end{document}